\documentclass[a4paper, 11pt]{article}

\usepackage{a4wide}
\usepackage[english]{babel}
\usepackage{amsmath}
\usepackage{psfrag}
\usepackage[latin1]{inputenc}
\usepackage[T1]{fontenc}
\usepackage{amsthm} 
\usepackage{amssymb}
\usepackage{verbatim}
\usepackage[usenames]{color}
\usepackage{stmaryrd}

\DeclareMathOperator{\res}{Res}
\DeclareMathOperator{\ind}{Ind}
\DeclareMathOperator{\Soc}{Soc}
\DeclareMathOperator{\Rad}{Rad}

\DeclareMathOperator{\End}{End}
\DeclareMathOperator{\Hom}{Hom}

\DeclareMathOperator{\Inf}{Inf}
\DeclareMathOperator{\GL}{GL}
\DeclareMathOperator{\id}{id}
\DeclareMathOperator{\im}{im}
\newcommand{\Liep}{\mathrm{Lie}_p}
\newcommand{\Lieppf}{\mathrm{Lie}^{\mathrm{pf}}_p}
\newcommand{\Lieppr}{\mathrm{Lie}^{\mathrm{pr}}_p}
\newcommand{\Lieppbl}{\mathrm{Lie}^{\mathrm{pbl}}_p}
\newcommand{\LieF}{\mathrm{Lie}_F}
\newcommand{\LieFpf}{\mathrm{Lie}^{\mathrm{pf}}_F}
\newcommand{\LieFpr}{\mathrm{Lie}^{\mathrm{pr}}_F}
\renewcommand{\leq}{\leqslant}
\renewcommand{\geq}{\geqslant}
\renewcommand{\unlhd}{\trianglelefteqslant}

\newcommand{\Liet}{\mathrm{Lie}_2}
\newcommand{\Liepft}{\mathrm{Lie}^{\mathrm{pf}}_2}
\newcommand{\Liepblt}{\mathrm{Lie}^{\mathrm{pbl}}_2}
\newcommand{\Lieth}{\mathrm{Lie}_3}
\newcommand{\Liepfth}{\mathrm{Lie}^{\mathrm{pf}}_3}
\newcommand{\Liepblth}{\mathrm{Lie}^{\mathrm{pbl}}_3}
\newcommand{\tildeM}{\widetilde{M}}

\newcommand{\hatN}{\widehat{N}}

\begin{document}
\swapnumbers
\theoremstyle{definition}
\newtheorem{defi}{Definition}[section]
\newtheorem{rem}[defi]{Remark}
\newtheorem{ques}[defi]{Question}
\newtheorem{expl}[defi]{Example}
\newtheorem{conj}[defi]{Conjecture}
\newtheorem{claim}[defi]{Claim}
\newtheorem{nota}[defi]{Notation}
\newtheorem{noth}[defi]{}

\theoremstyle{plain}
\newtheorem{prop}[defi]{Proposition}
\newtheorem{lemma}[defi]{Lemma}
\newtheorem{cor}[defi]{Corollary}
\newtheorem{thm}[defi]{Theorem}

\renewcommand{\proofname}{\textsl{\textbf{Proof}}}

\begin{center}
{\bf\Large Vertices of Lie Modules}

\medskip
Roger M.~Bryant, Susanne Danz, Karin Erdmann, and J\"urgen M\"uller\\

\today

\begin{abstract}
\noindent
Let $\LieF(n)$ be the Lie module of the symmetric group
$\mathfrak{S}_n$ over a field $F$ of characteristic $p>0$, that is,
$\LieF(n)$ is the left ideal of $F\mathfrak{S}_n$ generated by the
Dynkin--Specht--Wever element $\omega_n$. We study the problem of
parametrizing non-projective indecomposable summands of $\LieF(n)$, 
via describing their vertices and sources. Our main result shows
that this can be reduced to the case when $n$ is a power of $p$. 
When $n=9$ and $p=3$, and when $n=8$ and $p=2$, we present a precise answer.
This suggests a possible parametrization for arbitrary prime powers.

\smallskip

\noindent
{\bf Keywords:} Lie module, vertex, source, endo-permutation module, symmetric group

\noindent
{\bf MR Subject Classification:} 20C20, 20C30, 20G43

\end{abstract}
\end{center}

\section{Introduction}\label{sec:intro}

The Lie module of the symmetric group $\mathfrak{S}_n$ 
occurs in various contexts within algebra and topology,
where the name-giving property is its close relation to the
free Lie algebra; for more details, see for example the 
introduction in \cite{ET1}.
In the present paper, letting $F$ be an algebraically closed 
field of characteristic $p>0$, 
we realize the Lie module $\LieF(n)$ of $\mathfrak{S}_n$,
for $n\geq 2$, as the submodule $F\mathfrak{S}_n\omega_n$ 
of the regular $F\mathfrak{S}_n$-module, where 
$$ \omega_n:=(1-c_2)(1-c_3)\cdots (1-c_n)\in F\mathfrak{S}_n $$
is the Dynkin--Specht--Wever element of $F\mathfrak{S}_n$,
where in turn $c_k\in\mathfrak{S}_n$ is the backward cycle
$(k,k-1,\ldots,2,1)$.  

\begin{noth}
It is well known that $\omega_n^2=n\omega_n\in F\mathfrak{S}_n$.
Hence if $p$ does not divide $n$, then $\omega_n/n\in F\mathfrak{S}_n$
is an idempotent, so that $\LieF(n)$ is then a direct summand of
the regular $F\mathfrak{S}_n$-module and is, thus, projective. 
In the present paper we are interested in the case when $p$ divides $n$,
which we assume from now on in this section.
Then $\LieF(n)$ cannot be projective; for otherwise
$\dim(\LieF(n))=(n-1)!$ would have to be divisible by the $p$-part 
of $n!$, which is not the case.
Therefore, in this case $\LieF(n)$ admits a decomposition
\begin{equation*}\label{equ:liedec}
\LieF(n)=\LieFpr(n)\oplus \LieFpf(n),
\end{equation*}
where $\LieFpr(n)$ is a projective $F\mathfrak{S}_n$-module and where
$\LieFpf(n)\neq\{0\}$ is a projective-free $F\mathfrak{S}_n$-module. 

The asymptotic behaviour of the quotient 
$\dim(\LieFpr(n))/\dim(\LieF(n))$ has recently been studied by 
Erdmann--Tan \cite{ET1}, and by Bryant--Lim--Tan \cite{BLT}.
By \cite[Thm. 1.2]{BLT}, one has
$$\frac{\dim(\LieFpr(n))}{\dim(\LieF(n))}\longrightarrow 1,$$
as $n\longrightarrow \infty$ in $\mathbb{N}\smallsetminus\{p^k\mid k\geq 0\}$.
Moreover, it is conjectured in \cite{BLT} that this should remain true when
allowing $n$ to vary over all natural numbers.
This suggests that $\LieFpf(n)$ 
should be small, compared with the entire Lie module $\LieF(n)$.

Moreover, by work of Erdmann--Tan \cite{ET2}, we also know that 
the projective-free part $\LieFpf(n)$ of $\LieF(n)$ always belongs 
to the principal block of $F\mathfrak{S}_n$,
and Bryant--Erdmann \cite{BE} have studied indecomposable direct sum 
decompositions of the, necessarily projective, part of $\LieF(n)$
not contained in the principal block of $F\mathfrak{S}_n$. 
This leaves open, next to $\LieFpf(n)$, only the direct sum
decompositions of the component of $\LieFpr(n)$ 
belonging to the principal block of $F\mathfrak{S}_n$.
\end{noth}

\begin{noth}
One key ingredient of our approach is a decomposition theorem, 
expressing $\LieF(n)$ as a direct sum of pieces related to 
Lie modules $\LieF(p^d)$, for various $d$ such that $p^d$ divides $n$.
This is obtained by translating the 
Bryant--Schocker decomposition theorem \cite{BS}  for Lie powers
to Lie modules, using work of Lim--Tan \cite{LT}.
This paves the way to reduce questions on Lie modules to the 
case when $n$ is a power of $p$,
and puts the Lie modules $\LieF(p^d)$ into the focus of study. 
In particular, one is tempted to ask whether there is a neat description 
of the indecomposable direct summands of $\LieF(n)$ in terms of
those of $\LieF(p^d)$, where $d$ varies as indicated above.
This has been fully accomplished
for the case where $p$ divides $n$ but $p^2$ does not,
with a different line of reasoning, by Erdmann--Schocker \cite{ES},
while the general case remains a mystery and is subject to
further investigations.

Very little information concerning the decomposition 
of the principal block component of $\LieF(p^d)$ is available in the 
literature, and the projective-free part $\LieFpf(p^d)$ is very poorly 
understood, even for very small exponents $d$: to our knowledge, 
the only cases dealt with systematically are the modules
$\LieFpf(p)$, that is, the case $d=1$, by Erdmann--Schocker \cite{ES};
and, apart from the easy case $\LieF(4)=\LieFpf(4)$, 
there are just partial results for $\LieF(8)$, by Selick--Wu \cite{SW}.
The aim of this paper now is to investigate indecomposable direct summands
of $\LieF(p^d)$, for a few further small values of $p$ and $d$.

The major obstacle here
is that, due to the exponential growth of the dimension of Lie modules 
in terms of $n$, these modules quickly become very large.
Hence, to proceed further in this direction, we apply computational
techniques. More precisely, by this approach we are now able to give
a complete description of the Lie modules $\LieF(8)$ of dimension $5040$,
and $\LieF(9)$ of dimension $40320$.

Actually, in both cases it turns out that the projective-free part 
of the Lie module is already  indecomposable, where $\LieFpf(8)$
has dimension 816, and $\LieFpf(9)$ has dimension 1683.
In view of these results, and those on $\LieFpf(4)$ and 
$\LieFpf(p)$ mentioned above, the question 
arises whether $\LieFpf(p^d)$ is always indecomposable.
\end{noth}

\begin{noth}
To analyze the projective-free part of $\LieF(n)$, we are, in particular, interested in the Green vertices and sources
of the indecomposable direct summands of $\LieFpf(n)$. Using the reduction
result mentioned above, to some extent we are able to reduce this 
problem for arbitrary $n$ to the case where $n$ is a $p$-power.

Moreover, we are able to compute vertices and sources of $\LieFpf(8)$
and $\LieFpf(9)$. It turns out that both modules are endo-$p$-permutation
modules, in the sense of Urfer \cite{Urfer},
their vertices are regular elementary abelian subgroups of $\mathfrak{S}_8$ and $\mathfrak{S}_9$, respectively, 
and their sources are endo-permutation modules, 
in the sense of Dade \cite{Dade1}, whose class in the
Dade group we are able to determine.
It is surprising to us
to see the class of endo-permutation modules 
appear in this context. 

Hence, in view of these results, and those concerning 
$\LieFpf(4)$ and $\LieFpf(p)$, one may wonder whether 
$\LieFpf(p^d)$, assumed to be indecomposable,
always is an endo-$p$-permutation module 
having regular elementary abelian vertices
and endo-permutation sources, and, if so, 
what the class of a source in the Dade group looks like.
If this holds true, then, by our reduction results,
any indecomposable direct summand of any Lie module
will have vertices and sources sharing the same properties.
\end{noth}

\begin{noth}\label{noth:schedule}
This paper is organized as follows: 
in Section \ref{sec:pre} we provide the necessary prerequisites;
in particular, we recall the notions of Green vertices and sources,
endo-permutation and endo-$p$-permutation modules, and the Dade group.
In Section \ref{sec:lie} we introduce Lie modules; in order to make
this paper sufficiently self-contained, we also discuss Lie powers and 
their relation to Lie modules via the Schur functor, as well as 
variations on the construction and basic properties of Lie modules.

In Section \ref{sec:comp} we collect the explicit computational results 
we have obtained for specific examples; in particular, we present more 
details of the computational ideas and tools we have been using,
we revisit $\LieF(p)$ and $\LieF(4)$, and discuss the major examples 
$\LieF(8)$ and $\LieF(9)$, whose indecomposable direct sum decomposition
we determine, together with vertices and sources of the non-projective
indecomposable direct summands occurring.

In Section \ref{sec:reduction} we present a reduction,
eventually showing that vertices and sources of indecomposable 
direct summands of Lie modules in general can be described in
terms of the results in the $p$-power case; in order to do so,
in Theorem~\ref{thm:wreathreduce} we provide a description
of vertices and sources of indecomposable direct summands of
modules for wreath products,
in Theorem~\ref{prop:BE} we present the decomposition theorem
for Lie modules mentioned above, and in Theorem~\ref{thm:vertexLie} these
are combined to prove the reduction result.
\end{noth}

\medskip

\noindent
{\bf Acknowledgement:} The second author's research has been supported 
through a Marie Curie Career Integration Grant (PCIG10-GA-2011-303774).
The second and fourth authors have also been supported by the DFG Priority Programme `Representation Theory'
(grant \# DA1115/3-1). The first and third authors have been supported by EPSRC Standard Research Grant \# EP/G025487/1.

\section{Prerequisites}\label{sec:pre}

\begin{noth}\label{noth:general}
{\bf Generalities.}\,
(a)\, Throughout this paper, without further notice, we assume that
$F$ is  an algebraically closed field of characteristic $p>0$. 
Whenever $G$ is a finite group, an $FG$-module is always
understood to be a left $FG$-module of finite $F$-dimension, unless
stated otherwise. 
If $M$ and $N$ are $FG$-modules such that $N$ is 
isomorphic to a direct summand of $M$ then we write $N\mid M$.

Whenever $G$ is a finite group, $H$ is a normal subgroup of $G$, and
$M$ is an $F[G/H]$-module, we denote by $\Inf_{G/H}^G(M)$ the 
$FG$-module obtained from $M$ via inflation. 
More generally, by abuse of notation,
given a fixed epimorphism of groups $G\twoheadrightarrow K$ and an 
$FK$-module $M$, we denote the $FG$-module obtained via inflation
with respect to this epimorphism by $\Inf_K^G(M)$ as well.

\medskip

(b)\, 
By $\mathfrak{S}_n$ we denote the symmetric
group of degree $n\geq 1$, where permutations in $\mathfrak{S}_n$ are
also multiplied from right to left.
So, for instance, we have $(1,2)(2,3)=(1,2,3)$.

We assume the reader to be familiar with the basic notions of the 
representation theory of the symmetric group. 
For detailed background information, we refer to \cite{J,JK}.
The Specht modules $S^\lambda$ of the group algebra $F\mathfrak{S}_n$ will,
as usual, be labelled by the partitions $\lambda$ of $n$,
and the simple $F\mathfrak{S}_n$-modules $D^\lambda$ by
the $p$-regular partitions of $n$.
Furthermore, we denote by $P^\lambda$ a projective cover of $D^\lambda$.

\medskip

(c)\, Whenever $G$ is a finite group with subgroups $H$ and $K$, we write 
$H\leq_GK$ if $H$ is $G$-conjugate to a subgroup of $K$, and we write $H=_GK$ if
$H$ is $G$-conjugate to $K$.
\end{noth}

Next we recall the notions of vertices and sources of indecomposable modules
over group algebras, and we summarize some basic properties of
endo-permutation modules over finite $p$-groups. 
The latter class of modules has been introduced by Dade \cite{Dade1},
as generalizations of permutation modules. They have
proved to play an important role in modular representation theory
of finite groups, and, as we will see in subsequent sections, also
appear naturally in the context of Lie modules.
For a detailed account on
the theory of vertices and sources we refer the reader to \cite[Chap. 4.3]{NT}.
Background information concerning endo-permutation modules
can be found in \cite{Dade1,Dade2} and in \cite[\S 28]{Thev}.


\begin{noth}\label{noth:vertex}
{\bf Vertices and sources.}\,
(a)\, Let $G$ be a finite group, and let $M$ be an indecomposable 
$FG$-module. By Green's Theorem \cite{Green}, we can assign to $M$ a
$G$-conjugacy class of $p$-subgroups of $G$, the {\it vertices of $M$}.
A vertex $Q$ of $M$ is characterized by the property that $Q$ is minimal
such that $M$ is {\it relatively $Q$-projective}, 
that is, $M$ is isomorphic to a direct summand of $\ind_Q^G(N)$,
for some indecomposable $FQ$-module $N$. 
In particular, $M$ is projective if and only if $Q=\{1\}$.

Given a vertex $Q$ of $M$, an indecomposable $FQ$-module $L$ 
such that $M$ is isomorphic to a direct summand of $\ind_Q^G(L)$
is called a {\it $Q$-source} of $M$, and is unique up to
isomorphism and conjugation with elements in $N_G(Q)$.
Moreover, $Q$ is also a vertex of $L$.

\medskip

(b)\, Let $B$ be the block of $FG$ containing $M$. 
If $Q$ is a vertex of $M$ then there are a defect group
$R$ of $B$ and a Sylow $p$-subgroup $P$ of $G$ such that
$Q\leq R\leq P$. Moreover, as a consequence of
Green's Indecomposability Theorem \cite{Green}, 
$|P:Q|$ divides $\dim(M)$.

\medskip

(c)\, Suppose that $H\leq G$ is any subgroup of $G$ and that $N$ is an 
indecomposable direct summand of $\res_H^G(M)$ with vertex 
$R$ and $R$-source $L'$. Then there are a vertex $Q$ of $M$ 
and a $Q$-source $L$ of $M$ such that $R\leq Q$ and $L'\mid\res_R^Q(L)$.
This is seen as follows:

Let $Q$ be any vertex of $M$, and let $L$ be any $Q$-source of $M$. 
Then we have $L'\mid \res_R^H(N)\mid\res_R^G(M)$ and
$M\mid \ind_Q^G(L)$, thus $L'\mid \res_R^G(\ind_Q^G(L))$. 
Now, by Mackey's Theorem and the fact that $L'$ has vertex $R$,
this implies $L'\mid \res_R^{{}^gQ}({}^gL)$, for some $g\in G$ 
such that $R\leq {}^gQ$.
But ${}^gQ$ is also a vertex of $M$, and ${}^gL$ is a ${}^gQ$-source of $M$,
whence the claim.

In particular, if $H\leq G$ is such that $M$ is relatively $H$-projective,
then $M\mid \ind_H^G(\res_H^G(M))$ implies that there is
an indecomposable direct summand of $\res_H^G(M)$ sharing a 
vertex and a source with $M$. 

\medskip

(d)\, Suppose, conversely,  that $H\geq G$ is a finite overgroup of $G$
and that $N$ is an indecomposable direct summand of $\ind_G^H(M)$.
Then, given a vertex $Q$ of $M$ and a $Q$-source $L$ of $M$,
there is a vertex $P$ of $N$ such that $P\leq Q$, and there is some 
$P$-source of $N$ that is isomorphic to a direct summand of $\res_P^Q(L)$.
This is seen as follows:

Let $P$ be any vertex of $N$, and let $L'$ be any $P$-source of $N$.
Then, by Mackey's Theorem again, we get $L'\mid \res_P^{{}^hQ}({}^hL)$, 
for some $h\in H$ such that $P\leq {}^hQ$. In other words, we have
${}^{h^{-1}}L'\mid \res_{{}^{h^{-1}}P}^{Q}(L)$, where
${}^{h^{-1}}P$ is also a vertex of $N$, and ${}^{h^{-1}}L'$ 
is a ${}^{h^{-1}}P$-source of $N$, whence the claim.
(Note that we cannot conclude that ${}^{h^{-1}}L'$ is an
arbitrary ${}^{h^{-1}}P$-source of $N$, since the conjugating element $h$ 
 might depend on the choice of $L'$.)

In particular, since $M\mid\res_G^H(\ind_G^H(M))$, there is some 
indecomposable direct summand of $\ind_G^H(M)$ sharing a
vertex and a source with $M$.
\end{noth}


\begin{noth}\label{noth:endoperm}
{\bf Endo-permutation modules.}\,
An $FG$-module $M$, where $G$ is a finite group,
is called an {\it endo-permutation module} if its $F$-endomorphism ring 
$\End_F(M)\cong M\otimes M^*$ is a permutation $FG$-module. 
We list some properties of endo-permutation modules
that we will need later in this paper:

\medskip

(a)\, 
Permutation modules are endo-permutation modules. 
The class of endo-permutation modules is 
closed under taking $F$-linear duals, direct summands, tensor products,
restriction to subgroups, inflation from factor groups, and 
taking Heller translates $\Omega$ and $\Omega^{-1}$,
but it is neither closed under taking direct sums, 
nor under induction to finite overgroups.
In particular, any indecomposable endo-permutation module has 
endo-permutation sources.

\medskip

(b)\, The problem of classifying the 
indecomposable endo-permutation modules for finite $p$-groups $P$
has been worked on by various people.
The final classification result was obtained by Bouc \cite{Bouc},
but when $P$ is abelian, the following classification 
result is already due to Dade \cite{Dade2}. We will describe this result below; this is the version
we will need.

Note that it is indeed sufficient to classify the
indecomposable endo-permutation $FP$-modules with vertex $P$, 
since if $M$ is an indecomposable endo-permutation $FP$-module with vertex $Q<P$ 
then any $Q$-source $S$ of $M$ is an endo-permutation $FQ$-module 
with vertex $Q$, and by Green's Indecomposability Theorem \cite{Green}
we have $M\cong\ind_Q^P(S)$.
\end{noth}

\begin{thm}[\protect{\cite[Thm. 12.5]{Dade2}}]\label{thm:Dade}
Let $P$ be an abelian $p$-group, and let $M$ be an indecomposable
endo-permutation $FP$-module with vertex $P$. Then $M$ is, up to
isomorphism, the unique indecomposable direct summand of 
\begin{equation*}\label{equ:dade}
\bigotimes_{|P:Q|\geq 3}\Inf_{P/Q}^P(\Omega^{n_Q}(F_{P/Q}))
\end{equation*}
having vertex $P$. Here $n_Q\in\mathbb{Z}$, for $Q<P$, is uniquely
determined by $M$ if $P/Q$ is non-cyclic; otherwise $n_Q$ is
uniquely determined modulo $2$.
\end{thm}

\medskip
In other words, the isomorphism types of indecomposable
endo-permutation $FP$-modules with \emph{abelian} 
vertex $P$ are in bijection with
the elements of the \emph{Dade group}
$$ \mathfrak{D}(P)\cong
\left(\sum_{|P:Q|\geq 3,\, P/Q\text{ non-cyclic}}\mathbb{Z}\right)\oplus
\left(\sum_{|P:Q|\geq 3,\,P/Q\text{ cyclic}} \mathbb{Z}/2\mathbb{Z}
\right) \,. $$
For a precise definition of the Dade group of an arbitrary $p$-group $P$ and further details, see \cite{Dade1}
and \cite[\S 29]{Thev}. Whenever $S$ is an indecomposable endo-permutation $FP$-module with vertex
$P$, its image in $\mathfrak{D}(P)$ will be denoted by $[S]$. If $S$ and $S'$ are indecomposable endo-permutation $FP$-modules
with vertex $P$ then their tensor product $S\otimes S'$ has a unique (up to isomorphism) indecomposable direct summand $T$
with vertex $P$, and the multiplication in $\mathfrak{D}(P)$ is then given by $[S]\cdot [S']:=[T]$.

\begin{noth}\label{noth:endopperm}
{\bf Endo-$p$-permutation modules.}\,
According to Urfer \cite{Urfer,UrferThesis}, one can weaken the notion of 
endo-permutation modules as follows:
an $FG$-module $M$, where $G$ is a finite group, is called 
an {\it endo-$p$-permutation module}, if its $F$-endomorphism ring
$\End_F(M)\cong M\otimes M^*$ is a $p$-permutation $FG$-module,
that is, all its indecomposable direct summands are trivial-source modules.

\medskip

(a)\,
Again, one has the following standard properties of endo-$p$-permutation
modules: endo-permutation modules are endo-$p$-permutation modules.
The class of endo-$p$-permutation modules is
closed under taking $F$-linear duals, Heller translates, direct summands, tensor products,
restriction to subgroups, and inflation from factor groups,
but it is neither closed under taking direct sums, 
nor under induction to finite overgroups. Note that for finite 
$p$-groups the classes of endo-$p$-permutation modules and of 
endo-permutation modules coincide, thus any indecomposable
endo-$p$-permutation module has endo-permutation sources.

\medskip

(b)\,
A characterization of indecomposable endo-$p$-permutation module in terms
of vertices and sources is given in \cite[Thm. 1.5]{Urfer}. For the
cases of interest in the present paper it can be rephrased as follows:

Let $P\leq G$ be a $p$-group. As before, for any indecomposable endo-permutation $FP$-module $S$ with vertex $P$ let
$[S]\in\mathfrak{D}(P)$ be the associated element of the Dade group.
Then $[S]\in\mathfrak{D}(P)$ is called \emph{$G$-stable} if 
$$ \res^{P}_{P\cap {}^gP}([S])=\res^{{}^gP}_{P\cap {}^gP}([{}^gS])
   \in\mathfrak{D}(P\cap {}^gP)\,, 
\quad\text{for all}\quad g\in G \,. $$

Then, by \cite[Thm. 1.5]{Urfer},
an indecomposable $FG$-module with vertex $P$ and $P$-source $S$
is an endo-$p$-permutation module if and only if 
$S$ is an endo-permutation module such that $[S]$ is $G$-stable.

\medskip

(c)\,
As in part~(b), let $P\leq G$ be a $p$-group.
In good situations the 
$G$-stable elements of $\mathfrak{D}(P)$ are described by a 
Burnside-type fusion argument as follows:

Let $N_G(P)$ \emph{control fusion} in $P$, that is, whenever $Q\leq P$ and
$g\in G$ are such that ${}^gQ\leq P$,  there are some
$h\in N_G(P)$ and $z\in C_G(Q)$ such that $g=hz$.
Then, by \cite[La. 1.8, Prop. 1.9]{Urfer}, an element $[S]\in\mathfrak{D}(P)$
is $G$-stable if and only if it is fixed by the conjugation 
action of $N_G(P)$ on $\mathfrak{D}(P)$.
(Note that we do not require an additional saturation condition here,
as is done in \cite{Urfer}: an inspection of the 
proofs of \cite[La. 1.8, Prop. 1.9]{Urfer} shows that 
they are valid under the assumptions made here.)

Moreover, if $P$ is {\it abelian} then an element $[S]\in\mathfrak{D}(P)$ is an $N_G(P)$-fixed point
if and only if the associated function $Q\longrightarrow n_Q$ 
is constant on the $N_G(P)$-orbits on $\{Q<P\mid |P:Q|\geq 3\}$.
\end{noth}

\section{The Lie Module of the Symmetric Group}\label{sec:lie}

We begin this section by introducing the Lie module for $F\mathfrak{S}_n$,
we list some of its properties, and briefly discuss variations.
Many of these observations are certainly well known 
to the experts, but explicit references are not too easy to find. 
Thus we recall them here for the readers' convenience, 
and to make this paper as self-contained as possible. 

\begin{noth}\label{noth:E}
{\bf Lie powers.}\,
(a)\,
Let $\GL_n(F)$ be the general linear group over $F$, where $n\geq 1$,
which acts naturally on $F^n$, and let $\{e_1,\ldots,e_n\}$
be the standard basis of $F^n$. 
We may view $\mathfrak{S}_n$ as a subgroup of $\GL_n(F)$, by identifying 
a permutation $\pi\in\mathfrak{S}_n$ with the corresponding permutation matrix
in $\GL_n(F)$.

The $r$-th tensor power $(F^n)^{\otimes r}$, where $r\geq 1$,
is an $F[\GL_n(F)]$-module by way of the diagonal action.
Thus, via restriction $(F^n)^{\otimes r}$ becomes an 
$F\mathfrak{S}_n$-module,
where the symmetric group $\mathfrak{S}_n$ acts by substitutions, that is,
\begin{equation*}\label{equ:sub}
\pi: v_1\otimes\cdots\otimes v_r\longmapsto
\pi v_1\otimes\cdots\otimes \pi v_r,
\quad\text{for}\quad v_1,\ldots,v_r\in F^n,\; \pi\in\mathfrak{S}_n.
\end{equation*}

On the other hand, $(F^n)^{\otimes r}$ also carries a right 
$F\mathfrak{S}_r$-action `$*$' via place permutations,
which hence centralizes the $F[\GL_n(F)]$-action:
\begin{equation*}\label{equ:place}
\sigma: v_1\otimes\cdots\otimes v_r\longmapsto 
(v_1\otimes\cdots\otimes v_r)*\sigma=
v_{\sigma(1)}\otimes\cdots\otimes v_{\sigma(r)},
\quad\text{for}\quad v_1,\ldots,v_r\in F^n,\; 
\sigma\in\mathfrak{S}_r.
\end{equation*}

\medskip

(b)\,
Now we consider the {\it Lie bracket}
$$\kappa_2: (F^n)^{\otimes 2}\longrightarrow (F^n)^{\otimes 2},\; 
  v_1\otimes v_2\longmapsto [v_1,v_2]:=v_1\otimes v_2-v_2\otimes v_1, 
\quad\text{for}\quad v_1,v_2\in F^n. $$
By iteration, this yields the 
({\it left-normed}) Lie bracket
$$\kappa_r: (F^n)^{\otimes r}\longrightarrow (F^n)^{\otimes r},\; 
  v_1\otimes\cdots\otimes v_r\longmapsto 
[[v_1,v_2,\ldots,v_r]:=[\cdots [[v_1,v_2],v_3],\ldots,v_r] \,,$$
for all $r\geq 2$, and for completeness we  also let $\kappa_1:=\id$.

Hence we have $\kappa_r\in\End_F((F^n)^{\otimes r})$, for all $r\geq 1$,
where we assume $\End_F((F^n)^{\otimes r})$ to act on $(F^n)^{\otimes r}$
from the right, the action also being denoted by `$*$'.
The image $(v_1\otimes\cdots\otimes v_r)\ast\kappa_r\in (F^n)^{\otimes r}$
of a pure tensor $v_1\otimes\cdots\otimes v_r\in (F^n)^{\otimes r}$ is 
called an {\it (iterated) Lie bracket} of {\it length} $r$.

The Lie bracket induces the structure of a Lie algebra on the 
tensor algebra $T(F^n):=\bigoplus_{r\geq 1}(F^n)^{\otimes r}$.
Hence, by definition, for all $r\geq 2$ we have 
$(v_1\otimes\cdots\otimes v_r)\ast\kappa_r
=[(v_1\otimes\cdots\otimes v_{r-1})*\kappa_{r-1},v_r]$.
Moreover, for $r\geq 1$, the right adjoint action of $T(F^n)$ translates into
\begin{equation}\label{equ:adjact}
(v\otimes(v_1\otimes\cdots\otimes v_r)\ast\kappa_r)\ast\kappa_{r+1}
=[v,(v_1\otimes\cdots\otimes v_r)\ast\kappa_r]\in (F^n)^{\otimes (r+1)}.
\end{equation}

The map $\kappa_r$ centralizes the $F[\GL_n(F)]$-action,
hence the image 
$$L^r(F^n):=(F^n)^{\otimes r}*\kappa_r\subseteq (F^n)^{\otimes r}$$ 
of $\kappa_r$ is an $F[\GL_n(F)]$-submodule of $(F^n)^{\otimes r}$,
being called the $r$-th {\it Lie power} of $F^n$, where of course
we have $L^1(F^n)=F^n*\kappa_1=F^n*\id=F^n$. Thus we obtain the
free Lie algebra on $\{e_1,\ldots,e_n\}$
$$ L(F^n):=\bigoplus_{r\geq 1}L^r(F^n)\subseteq T(F^n)\, .$$
The fact that $L(F^n)$ is free as a Lie algebra is well known, and is due to Witt.

\medskip
(c)\,
By Schur--Weyl duality, 
the action of $\kappa_r$ is induced by the place permutation action 
of some  element 
$\omega_r\in F\mathfrak{S}_r$, which we are now going to determine:

For $r\geq 1$ let $c_r:=(r,r-1,\ldots,1)\in\mathfrak{S}_r$.
Note that, of course, $c_1=1$.
Then the place permutation action yields
\begin{equation*}\label{equ:cyc}
(v_1\otimes\cdots\otimes v_r)*c_r=
(v_r\otimes v_1\otimes\cdots\otimes v_{r-1}),
\quad\text{for}\quad v_1,\ldots,v_r\in F^n\,.
\end{equation*}

Now, for $r=2$, we have
$(v_1\otimes v_2)*\kappa_2=[v_1,v_2]=v_1\otimes v_2-v_2\otimes v_1
=(v_1\otimes v_2)*(1-c_2)$,
while for $r\geq 3$ and $v_1,\ldots,v_r\in F^n$ we get
\begin{align*}
(v_1\otimes\cdots\otimes v_r)*\kappa_r &=
((v_1\otimes\cdots\otimes v_{r-1})*\kappa_{r-1})\otimes v_r
-v_r\otimes ((v_1\otimes\cdots\otimes v_{r-1})*\kappa_{r-1}) \\
&=
(v_1\otimes\cdots\otimes v_r)*(\kappa_{r-1}\otimes\id)
-(v_1\otimes\cdots\otimes v_r)*(\kappa_{r-1}\otimes\id)*c_r \\
&=
(v_1\otimes\cdots\otimes v_r)*(\kappa_{r-1}\otimes\id)*(1-c_r) .
\end{align*}
Thus, by induction on $r\geq 2$, this gives
$\kappa_r= {}*(1-c_2)*(1-c_3)*\cdots *(1-c_r):
(F^n)^{\otimes r}\longrightarrow (F^n)^{\otimes r}$,
so that, for $r\geq 2$, we have
$$ \omega_r:=(1-c_2)(1-c_3)\cdots (1-c_r)\in F\mathfrak{S}_r\, ,$$
and 
$L^r(F^n)=(F^n)^{\otimes r}*\kappa_r=(F^n)^{\otimes r}*\omega_r$. 
The element $\omega_r$ is called
the {\it Dynkin--Specht--Wever element} of $F\mathfrak{S}_r$;
for completeness, since $\kappa_1=\id$, we let 
$\omega_1:=c_1\in F\mathfrak{S}_1$. Note that we even have
$\omega_r\in\mathbb{F}_p\mathfrak{S}_r$, where $\mathbb{F}_p$
is the prime field of $F$.
\end{noth}

\begin{noth}\label{noth:Lie}
{\bf Lie modules and the Schur functor.}\,
(a)\, Now let $n\geq r$. Then the {\it classical Schur functor} $\mathfrak{W}^r$ takes 
homogeneous polynomial $F[\GL_n(F)]$-modules of degree $r$ to
$F\mathfrak{S}_r$-modules, where, more precisely, an $F[\GL_n(F)]$-module 
$V$ is mapped to its $(1^r)$-weight space $\mathfrak{W}^{r}(V)$. In particular, for the 
$F[\GL_n(F)]$-module $(F^n)^{\otimes r}$ one more explicitly gets the following:

As mentioned above, the natural $\GL_n(F)$-action on $(F^n)^{\otimes r}$ 
induces a permutation action of $\mathfrak{S}_n$, and thus also a permutation action of
$\mathfrak{S}_r$, on $(F^n)^{\otimes r}$.
The vector $e_1\otimes\cdots\otimes e_r\in (F^n)^{\otimes r}$
affords a regular $\mathfrak{S}_r$-orbit
and, hence, induces an embedding of the regular 
$F\mathfrak{S}_r$-module into $(F^n)^{\otimes r}$ via
$$F\mathfrak{S}_r\longrightarrow (F^n)^{\otimes r},\; 
\pi\longmapsto \pi e_1\otimes\cdots\otimes \pi e_r
=e_{\pi(1)}\otimes\cdots\otimes e_{\pi(r)},
\quad\text{for}\quad \pi\in\mathfrak{S}_r .$$
The image of this embedding equals the 
$(1^r)$-weight space 
$$ \mathfrak{W}^{r}((F^n)^{\otimes r})
=\mathrm{Span}_F(\{e_{\pi(1)}\otimes\cdots\otimes e_{\pi(r)}
\mid \pi\in\mathfrak{S}_r\})\subseteq (F^n)^{\otimes r} $$
of the $F[\GL_n(F)]$-module $(F^n)^{\otimes r}$.
Moreover, the place permutation action of $\mathfrak{S}_r$ on
$(F^n)^{\otimes r}$ restricts to $\mathfrak{W}^{r}((F^n)^{\otimes r})$,
and via the above isomorphism
$\mathfrak{W}^{r}((F^n)^{\otimes r})\cong F\mathfrak{S}_r$
translates into right multiplication on $F\mathfrak{S}_r$.

\smallskip

\mbox{}From now on, suppose that $n=r$, which will be the case most relevant to
us.

\medskip

(b)\,
Now one defines the {\it Lie module} $\LieF(n)$ of $F\mathfrak{S}_n$ as 
the $(1^n)$-weight space of the $n$-th Lie power $L^n(F^n)$, that is,
$$ \LieF(n):=\mathfrak{W}^{n}(L^n(F^n))\subseteq (F^n)^{\otimes n}\, .$$
\mbox{}From $L^n(F^n)=(F^n)^{\otimes n}*\kappa_n\subseteq (F^n)^{\otimes n}$ one
thus gets
$$\LieF(n)=\mathfrak{W}^{n}((F^n)^{\otimes n}*\kappa_n)
=\mathfrak{W}^{n}((F^n)^{\otimes n})*\kappa_n
=\mathrm{Span}_F(\{[[e_{\pi(1)},\ldots,e_{\pi(n)}]\mid
 \pi\in\mathfrak{S}_n\}).  $$
Via the isomorphism 
$\mathfrak{W}^{n}((F^n)^{\otimes n})\cong F\mathfrak{S}_n$ 
of $F\mathfrak{S}_n$-modules, 
$\LieF(n)$ can be regarded as a submodule of the regular 
$F\mathfrak{S}_n$-module $F\mathfrak{S}_n$.
Since the action of $\kappa_n$ is induced by the place permutation action
of $\omega_n$,
we get
\begin{equation}\label{equ:LieF}
\LieF(n)\cong F\mathfrak{S}_n\cdot \omega_n \subseteq F\mathfrak{S}_n.
\end{equation}
Note that, in particular, $\LieF(1)\cong F$, the trivial
$F\mathfrak{S}_1$-module.

Moreover, we observe that $\LieF(n)$ is already realized over the
prime field $\mathbb{F}_p$ of $F$, that is, letting 
$$\Liep(n):=\mathbb{F}_p\mathfrak{S}_n\cdot\omega_n\subseteq 
  \mathbb{F}_p\mathfrak{S}_n 
\quad\text{as $\mathbb{F}_p\mathfrak{S}_n$-modules} ,$$
we get $\LieF(n)\cong F\otimes_{\mathbb{F}_p}\Liep(n)$
as $F\mathfrak{S}_n$-modules. We will make use of this in order
to facilitate explicit computations in Section \ref{sec:comp}.
\end{noth}

\begin{noth}\label{noth:variations}
{\bf Variations on Lie modules.}\,
Since there also exist slight modifications of the above modules
in the literature, we briefly comment on variations
of the construction:

\medskip
(a)\,
Firstly, starting with another vector 
$e_{\pi(1)}\otimes\cdots\otimes e_{\pi(n)}\in (F^n)^{\otimes n}$,
where $\pi\in\mathfrak{S}_n$, leads to
a different identification of $\LieF(n)$ with a submodule
of $F\mathfrak{S}_n$, namely to  the $F\mathfrak{S}_n$-module 
$F\mathfrak{S}_n\cdot \pi\omega_n\pi^{-1}$, that is,
amounts to a renumbering.

\medskip
(b)\,
Secondly, taking right-normed Lie brackets instead, for $r\geq 2$ one gets
$$ \kappa'_r: (F^n)^{\otimes r}\longrightarrow (F^n)^{\otimes r},\;
v_1\otimes\cdots\otimes v_r\longmapsto [v_1,v_2,\ldots, v_r]] :=
   [v_1,[v_2,\ldots,[v_{r-1,}v_r]\cdots]] ;$$
we again let $\kappa'_1:=\id$.
Since $[v_1,v_2,\ldots, v_r]]=(-1)^{r-1} \cdot [[v_r,v_{r-1},\ldots, v_1]$,
for $r\geq 1$, we get
$\kappa'_r=(-1)^{r-1}*w_r*\kappa_r: 
(F^n)^{\otimes r}\longrightarrow (F^n)^{\otimes r}$,
where $w_r\in\mathfrak{S}_r$ is the longest element of 
$\mathfrak{S}_r$ in the Coxeter sense, that is,
$$ w_r=(1,r)(2,r-1)\cdots $$
Thus this construction yields
$ \LieF'(n) = \mathfrak{W}^n((F^n)^{\otimes n})* \kappa'_n = \mathfrak{W}^n((F^n)^{\otimes n})*w_n*\kappa_n
= \mathfrak{W}^n((F^n)^{\otimes n})*\kappa_n = \LieF(n) .$

\medskip
(c)\,
Lastly, we analyze the $F\mathfrak{S}_n$-module 
$F\mathfrak{S}_n\cdot \omega_n^\iota\subseteq F\mathfrak{S}_n$, where
$$ \omega_n^\iota:=(1-c_n^{-1})(1-c_{n-1}^{-1})\cdots (1-c_2^{-1})
\in F\mathfrak{S}_n $$ 
is the image of $\omega_n$ under the $F$-algebra anti-automorphism 
$\iota:F\mathfrak{S}_n\longrightarrow F\mathfrak{S}_n$ defined
by $\iota:\pi\longmapsto \pi^{-1}$, for $\pi\in\mathfrak{S}_n$.
Then we have an isomorphism of left $F\mathfrak{S}_n$-modules
$$F\mathfrak{S}_n\cdot \omega_n^\iota\cong 
(F\mathfrak{S}_n\cdot\omega_n)^*,$$
where the latter $F\mathfrak{S}_n$-module denotes the contragredient
dual of $F\mathfrak{S}_n\cdot\omega_n$.
This fact is most elegantly established by recalling that
group algebras are, in particular, symmetric algebras, and using 
the general isomorphism (\ref{eqn symm alg}) in
Remark~\ref{rem:symm} below. Thus we briefly deviate to establish this:
\end{noth}

\begin{rem}\label{rem:symm}
Let $A$ be a finite-dimensional
symmetric $F$-algebra with symmetrizing $F$-bilinear form 
$\langle\cdot|\cdot\rangle$. That is, $\langle\cdot|\cdot\rangle$ is 
associative, symmetric, and non-degenerate. Hence 
\begin{equation*}\label{equ:symmiso}
A\longrightarrow \Hom_F(A,F),\; 
a\longmapsto (b\longmapsto \langle b|a\rangle),
\quad\text{for}\quad a,b\in A,
\end{equation*}
is an isomorphism of $(A,A)$-bimodules.
Letting $\omega\in A$,
this induces an isomorphism
$$\Hom_F(A\omega,F)\cong A/(A\omega)^\perp$$
of right $A$-modules,
where ${}^\perp$ denotes taking orthogonal spaces with respect to 
$\langle\cdot|\cdot\rangle$.
Moreover, 
$$(A\omega)^\perp
= \{a\in A\mid \langle A\omega|a\rangle=0\}
=\{a\in A\mid \langle A|\omega a\rangle=0\}
=\{a\in A\mid \omega a=0\}
=\mathrm{ker}(\omega\, \cdot\, ). $$
Since 
$A/\mathrm{ker}(\omega\,\cdot\,)\cong\im(\omega\,\cdot\,)=\omega A$, 
this yields an isomorphism of right $A$-modules
\begin{equation*}\label{equ:isoright}
\Hom_F(A\omega,F)\cong A/(A\omega)^\perp=
A/\mathrm{ker}(\omega\;\cdot\;)\cong \im(\omega\,\cdot\,)=\omega A.
\end{equation*}
Finally, suppose that there is an involutory $F$-algebra anti-automorphism 
$\iota:A\longrightarrow A,\;a\longmapsto a^\iota$. 
Then, whenever $M$ is a right $A$-module, one can define a left 
$A$-module structure on $M$ by
$a\cdot x:=xa^\iota$, for $x\in M,\, a\in A$.
Denoting the resulting module by $M^\iota$, one, in particular, gets
$(\omega A)^\iota\cong A\omega^\iota$ as left $A$-modules.
Thus one has an isomorphism of left $A$-modules
\begin{equation}\label{eqn symm alg}
\Hom_F(A\omega,F)^\iota \cong (\omega A)^\iota\cong A\omega^\iota \,.
\end{equation}
\end{rem}

\begin{noth}\label{noth:lieproperties}
{\bf Properties of Lie modules.}\,
We collect a couple of properties of Lie modules. 

\medskip
(a)\,
We exhibit an explicit $F$-basis of $\LieF(n)$.
Firstly, any element of $\LieF(n)$
is an $F$-linear combination of Lie brackets of the form
$[[e_n,e_{\pi(1)},\ldots,e_{\pi(n-1)}]$, 
where $\pi\in\mathfrak{S}_{n-1}$. This is clear for $n\leq 2$
anyway, and for $n\geq 3$ is seen as follows: letting $w$ be a Lie bracket 
involving a subset of $\{e_1,\ldots,e_{i-1},e_{i+1},\ldots,e_{n-1}\}$,
where $1\leq i\leq n-1$, we have $[[w,e_i,e_n]=[e_n,[e_i,w]]$,
where, by (\ref{equ:adjact}), the latter can be written as a sum of
Lie brackets having $e_n$ as their first component.

Now, since expanding $[[e_n,e_{\pi(1)},\ldots,e_{\pi(n-1)}]$
into the standard $F$-basis of $(F^n)^{\otimes n}$
yields a unique summand having $e_n$ as its first component,
namely $e_n\otimes e_{\pi(1)}\otimes\cdots\otimes e_{\pi(n-1)}$,
we conclude that 
$$ \{[[e_n,e_{\pi(1)},\ldots,e_{\pi(n-1)}]\mid \pi\in\mathfrak{S}_{n-1}\} 
   \subseteq \LieF(n) $$
is $F$-linearly independent, thus is an $F$-basis. Moreover, since 
$$ [[e_n,e_{\pi(1)},\ldots,e_{\pi(n-1)}]
=\pi\cdot[[e_n,e_1,\ldots,e_{n-1}]
=\pi\cdot c_n\cdot [[e_1,e_2,\ldots,e_n]
\quad\text{for}\quad\pi\in\mathfrak{S}_{n-1} ,$$
the above $F$-basis can also be written as
$$ \{\pi\cdot c_n\cdot [[e_1,\ldots,e_n]\mid \pi\in\mathfrak{S}_{n-1}\} 
   \subseteq \LieF(n) .$$

Thus $\res_{\mathfrak{S}_{n-1}}^{\mathfrak{S}_n}(\LieF(n))$ 
is isomorphic to the regular module $F\mathfrak{S}_{n-1}$,
in particular saying that $\dim_F(\LieF(n))=(n-1)!$.
Moreover, since 
$[[e_1,\ldots,e_n]=(e_1\otimes\cdots\otimes e_n)*\omega_n\in\LieF(n)$
is sent to $\omega_n\in F\mathfrak{S}_n$ 
via the isomorphism (\ref{equ:LieF}),
this means that an $F$-basis of $F\mathfrak{S}_n\cdot\omega_n$ 
is obtained as
$$ \{\pi\cdot c_n\cdot\omega_n\mid \pi\in\mathfrak{S}_{n-1}\} .$$
This $F$-basis will be particularly useful to facilitate the explicit
computations in Section \ref{sec:comp}; note that, by the observations
in \ref{noth:Lie}, this is even an $\mathbb{F}_p$-basis
of $\Liep(n)=\mathbb{F}_p\mathfrak{S}_n\cdot\omega_n$.

\medskip
(b)\,
We now show that $\omega_n^2=n\cdot\omega_n\in F\mathfrak{S}_n$,
in particular implying that $\frac{1}{n}\cdot\omega_n\in F\mathfrak{S}_n$
is an idempotent whenever $p\nmid n$. We proceed in various steps: 

Firstly, we show that for $n\geq 1$ we have
$\omega_n\cdot (e_1\otimes\cdots\otimes e_n)
  =[[e_1,\ldots,e_n]\in (F^n)^{\otimes n}$:
this is clear for $n=1$ anyway, and for $n=2$ we have
$\omega_2\cdot (e_1\otimes e_2)=(1-c_2)\cdot(e_1\otimes e_2)
=e_1\otimes e_2-e_2\otimes e_1=(e_1\otimes e_2)*\kappa_2$.
For $n\geq 3$, arguing by induction and using 
$\omega_n=\omega_{n-1}\cdot(1-c_n)\in F\mathfrak{S}_n$, we get
\begin{align*}
\omega_n\cdot (e_1\otimes\cdots\otimes e_n) &=
\omega_{n-1}\cdot 
(e_1\otimes\cdots\otimes e_n-e_n\otimes e_1\otimes\cdots\otimes e_{n-1}) \\
&=((e_1\otimes\cdots\otimes e_{n-1})*\kappa_{n-1})\otimes e_n
-e_n\otimes((e_1\otimes\cdots\otimes e_{n-1})*\kappa_{n-1}) \\
&=[(e_1\otimes\cdots\otimes e_{n-1})*\kappa_{n-1},e_n] \\
&=(e_1\otimes\cdots\otimes e_n)*\kappa_n =[[e_1,\ldots,e_n]\,.
\end{align*}

Secondly, we show that for $n\geq 2$ we have
$ \omega_{n-1}c_n\cdot[[e_1,\ldots,e_n]
=-[[e_1,\ldots,e_n]\in (F^n)^{\otimes n}$:
Recall that $\kappa_n\in\End_F((F^n)^{\otimes n})$ 
centralizes the $F[\GL_n(F)]$-action.
Then, using (\ref{equ:adjact}) we get
\begin{align*}
\omega_{n-1}c_n\cdot(e_1\otimes\cdots\otimes e_n)*\kappa_n &=
(e_n\otimes\omega_{n-1}\cdot(e_1\otimes\cdots\otimes e_{n-1}))*\kappa_n \\
&=(e_n\otimes(e_1\otimes\cdots\otimes e_{n-1})*\kappa_{n-1})*\kappa_n \\
&=[e_n,(e_1\otimes\cdots\otimes e_{n-1})*\kappa_{n-1}] \\
&=-[(e_1\otimes\cdots\otimes e_{n-1})*\kappa_{n-1},e_n] \\
&=-(e_1\otimes\cdots\otimes e_n)*\kappa_n .
\end{align*}

Combining these computations we get
$$ \omega_{n-1}c_n\cdot\omega_n\cdot (e_1\otimes\cdots\otimes e_n)
= -\omega_n\cdot (e_1\otimes\cdots\otimes e_n) ,$$
thus translating via the isomorphism 
$\mathfrak{W}^{n}((F^n)^{\otimes n})\cong F\mathfrak{S}_n$ yields 
$\omega_{n-1}c_n\cdot \omega_n= -\omega_n\in F\mathfrak{S}_n$.

Thirdly, and finally, we show that 
$\omega_n^2=n\cdot\omega_n\in F\mathfrak{S}_n$, for $n\geq 1$:
this is clear for $n=1$ anyway, and for $n=2$ we have 
$\omega_2^2=(1-c_2)^2=\omega_2-c_2\omega_2=2\omega_2$. Then,
for $n\geq 3$ we have 
$\omega_n^2=\omega_{n-1}(1-c_n)\omega_n
=\omega_{n-1}\omega_n-\omega_{n-1}c_n\omega_n$,
where, by induction, the first summand equals
$$ \omega_{n-1}\omega_n=\omega_{n-1}^2(1-c_n)
=(n-1)\omega_{n-1}(1-c_n)=(n-1)\omega_n .$$
The second summand being $-\omega_{n-1}c_n\omega_n=\omega_n$, this yields
$\omega_n^2=(n-1)\omega_n+\omega_n =n\omega_n$.
\end{noth}

\section{Computational Data}\label{sec:comp}

In this section we summarize our computational results 
concerning the Lie modules for some symmetric groups.
Before doing so, we need a few preparations.

\begin{noth}\label{noth:sylow}
{\bf Some subgroups of symmetric groups.}\,
(a)\,
We will use the following convention for denoting the Sylow $p$-subgroups
of the symmetric group $\mathfrak{S}_n$: suppose first that $n=p^d$, for
some $d\geq 0$. Moreover, we set $P_1:=1$ and $P_p:=C_p$, where
$C_p:=\langle (1,2,\ldots,p)\rangle$, as well as
$$P_{p^{i+1}}:=P_{p^i}\wr C_p=\{(g_1,\ldots,g_p;\sigma)\mid
  g_1,\ldots,g_p\in P_{p^i},\, \sigma\in C_p\}
\quad\text{for}\quad i\geq 1\,.$$
Note that the multiplication in $P_{p^{i+1}}$ is as explained in (\ref{equ:wreath multi}) below.
For $i\geq 0$, we view $P_{p^i}$ as a subgroup of $\mathfrak{S}_{p^i}$ in the
obvious way. Then, by \cite[4.1.22, 4.1.24]{JK}, $P_n$
is a Sylow $p$-subgroup of $\mathfrak{S}_n$, which 
can be generated by the following elements in $\mathfrak{S}_n$:
$$g_j:=\prod_{k=1}^{p^{j-1}}(k,k+p^{j-1},k+2p^{j-1},\ldots, k+(p-1)p^{j-1})
\quad\text{where}\quad j=1,\ldots,d \,. $$

Next suppose that $p\mid n$, but $n$ is not necessarily
a $p$-power. Consider the $p$-adic expansion
$n=\sum_{j=1}^s \alpha_jp^{i_j}$, for some $s\geq 1$, $i_1>\ldots >i_s\geq 1$,
and $1\leq \alpha_j\leq p-1$ for $j=1,\ldots,s$.
By \cite[4.1.22, 4.1.24]{JK},
$P_n:=\prod_{j=1}^s \prod_{l_j=1}^{\alpha_j}P_{p^{i_j},l_j}$
is then a Sylow $p$-subgroup of $\mathfrak{S}_n$. Here, the direct factor
$P_{p^{i_1},1}$ is acting on $\{1,\ldots,p^{i_1}\}$, $P_{p^{i_1},2}$ is acting
on $\{p^{i_1}+1,\ldots,2p^{i_1}\}$, and so on.

If $n$ is not divisible by $p$, let $m<n$ be maximal such that $p\mid m$,
and set $P_n:=P_m$, so that $P_n$ is a Sylow $p$-subgroup of
$\mathfrak{S}_n$ also in this case.

\medskip

(b)\,
For $d\geq 1$ we denote by $E_{p^d}$ the unique maximal
elementary abelian subgroup of $P_{p^d}$ that acts regularly
on $\{1,\ldots,p^d\}$; in particular, $|E_{p^d}|=p^d$.  
Letting $n:=p^d$ and $Q:=E_{p^d}$, we determine the structure of 
$N_{\mathfrak{S}_n}(Q)$:

Since $Q\leq\mathfrak{S}_n$ is an abelian transitive
subgroup, it is self-centralizing, and thus 
$N_{\mathfrak{S}_n}(Q)/Q$ is isomorphic to a subgroup of $\GL_d(p)$. 
Moreover, since the affine linear group $\text{AGL}_d(p)\cong C_p^d\rtimes\GL_d(p)$
acts transitively and faithfully on its elementary abelian subgroup $C_p^d$,
there is an embedding $\text{AGL}_d(p)\longrightarrow\mathfrak{S}_n$,
mapping $C_p^d$ to $Q$. Hence we conclude that 
$$ N_{\mathfrak{S}_n}(Q)\cong Q\rtimes\GL_d(p) .$$

Thus $N_{\mathfrak{S}_n}(Q)$ acts transitively on each of 
the sets $\{R\leq Q\mid |R|=p^i\}$, for $0\leq i\leq d$, and the
stabilizer $N_{N_{\mathfrak{S}_n}(Q)}(R)$ induces the
full automorphism group on any subgroup $R\leq Q$.
Hence $N_{\mathfrak{S}_n}(Q)$ controls fusion in $Q$, in the sense
of \ref{noth:endopperm}.
\end{noth}

\begin{expl}\label{expl:sylow}
If $p=2$ and $n=4$ then 
\begin{align*}
P_4&=\langle (1,2), (1,3)(2,4) \rangle\quad\text{and}\\
E_4&=\langle (1,2)(3,4), (1,3)(2,4) \rangle.
\end{align*}

If $p=2$ and $n=8$ then 
\begin{align*}
P_8&=\langle (1,2), (1,3)(2,4), (1,5)(2,6)(3,7)(4,8) \rangle\quad\text{and}\\
E_8&=\langle (1,2)(3,4)(5,6)(7,8), (1,3)(2,4)(5,7)(6,8), (1,5)(2,6)(3,7)(4,8) 
     \rangle.
\end{align*}

If $p=3$ and $n=9$ then 
\begin{align*}
P_9&=\langle (1,2,3), (1,4,7)(2,5,8)(3,6,9)\rangle \quad\text{and}\\
E_9&=\langle (1,2,3)(4,5,6)(7,8,9),(1,4,7)(2,5,8)(3,6,9)\rangle.
\end{align*}
\end{expl}

\begin{rem}\label{rem:lievertex}
(a)\,
Before proceeding to computationally substantial examples,
for later use we recall the following well-known case:
by \cite[Cor. 9]{ES} the projective-free part 
$\LieFpf(p)$ of the Lie module $\LieF(p)$ is 
indecomposable, and actually isomorphic to the 
Specht module $S^{(p-1,1)}$ of dimension $p-1$, hence 
$$ \LieFpf(p) \cong \Omega(F) 
\quad\text{ as $F\mathfrak{S}_p$-modules} \, .$$
Hence $\LieFpf(p)$ has vertex $E_p = C_p$, of course,
and $\Omega(F)$ is an $E_p$-source, having dimension $p-1$.
Note that the $FE_p$-module $\Omega(F)$ is an endo-permutation module. 
Thus, from Theorem \ref{thm:Dade}
and the remarks in \ref{noth:endopperm} and \ref{noth:sylow}(b) 
we conclude that $\LieFpf(p)$ is an endo-$p$-permutation module.

\medskip
(b)\,
In view of the subsequent results, we ask ourselves whether $\LieFpf(p)$
itself possibly is an endo-permutation $F\mathfrak{S}_p$-module. 
Indeed, for $p=2$ we have $\LieF(2)\cong\LieFpf(2)\cong F$, hence
$\LieFpf(2)$ is even a permutation $F\mathfrak{S}_2$-module. 

For $p=3$ we have $\LieF(3)\cong\LieFpf(3)\cong\Omega(F)$,
and from the theory of blocks of cyclic defect it is immediate that
$$\LieFpf(3)\otimes\LieFpf(3)^*\cong\Omega(F)\otimes\Omega(F)^*
\cong F\oplus P^{(2,1)} \,,$$
where $D^{(2,1)}$ is the sign representation. Note that,
in accordance with part~(a), all indecomposable direct summands of
$\LieFpf(3)\otimes\LieFpf(3)^*$ are trivial-source modules indeed.

To show that $\LieFpf(3)\otimes\LieFpf(3)^*$ is not a permutation
$F\mathfrak{S}_3$-module, assume to the contrary that it is.
Thus, by dimension reasons we conclude that $P^{(2,1)}$ is an
indecomposable transitive permutation $F\mathfrak{S}_3$-module, but
$P^{(2,1)}$ does not have the trivial module as an epimorphic image,
a contradiction.
\end{rem}

\begin{noth}\label{noth:computations}
{\bf Computational approach.}\,
We now give a description of the tools from computational
group theory and computational representation theory 
we are employing, and indicate the computational ideas
we are using to obtain the subsequent explicit results
concerning some larger Lie modules.
As a general background reference, see \cite{LuxPah}.

\medskip
(a)\,
To deal with finite groups, in particular permutation groups and 
matrix groups, we use the general purpose computer algebra 
systems {\sf GAP} \cite{GAP} and {\sf MAGMA} \cite{MAGMA}.
In particular, we make use of the character table library 
{\sf CTblLib} \cite{CTblLib} of {\sf GAP}, which provides 
electronic access to the data collected in the {\sf Atlas} \cite{Atlas} 
and in the {\sf ModularAtlas} \cite{ModularAtlas,ModularAtlasProject};
these databases, in particular, contain the explicit $2$- and $3$-modular
decomposition matrices for various symmetric groups given in 
\cite[App.]{J} or \cite[App. I]{JK}.

Moreover, we have used the more specialized computer algebra system 
{\sf MeatAxe} \cite{Par,MA}, and its extensions 
\cite{LuxMueRin,LuxSzokeII,LuxSzoke,LuxWie}, to deal with
various aspects concerning matrix representations over (small) finite fields.
Apart from general linear algebra, these tools, in particular, allow us
to find composition series and direct sum decompositions, 
including isomorphism checks of simple and indecomposable
modules, respectively, and
to find splitting fields and to check absolute indecomposability;
moreover, they enable us
to compute homomorphism spaces and endomorphism rings, 
to determine radical and socle series, and to compute submodule lattices;
apart from these analytic capabilities, they
also provide the constructions needed below, such as Kronecker
products and the computation of Heller translates.

\medskip

(b)\,
To facilitate explicit computations, we make use of the observation
in \ref{noth:Lie}, saying that 
$\LieF(n)\cong F\otimes_{\mathbb{F}_p}\Liep(n)$
as $F\mathfrak{S}_n$-modules, where 
$$\Liep(n):=\mathbb{F}_p\mathfrak{S}_n\cdot\omega_n\subseteq 
  \mathbb{F}_p\mathfrak{S}_n 
\quad\text{as $\mathbb{F}_p\mathfrak{S}_n$-modules}.$$
Thus we are indeed reduced to considerations of permutation representations,
and matrix representations over finite (prime) fields.

Having got hands on the $\mathbb{F}_p\mathfrak{S}_n$-module
$\Liep(n)$, the task then is to find the
decomposition $\Liep(n)\cong\Lieppf(n)\oplus\Lieppr(n)$ into
its projective-free and projective part, respectively, 
to determine how $\Lieppr(n)$ decomposes into projective indecomposable
modules, and what the indecomposable direct summands of $\Lieppf(n)$ look like. 
However, the examples of Lie modules $\Liep(n)$ to be dealt with here are 
too large to simply apply to them the general techniques available 
to compute direct sum decompositions. Hence we have to proceed otherwise 
to find $\Lieppf(n)$ in the first place; after all, by the asymptotic results
mentioned in the introduction, we expect $\Lieppf(n)$ to be small 
compared with $\Liep(n)$, small enough to allow for a detailed analysis.

By \cite{ET2}, we know that, in order to detect $\Lieppf(n)$,
we only need to consider the component $\Lieppbl(n)$ of $\Liep(n)$ 
belonging to the principal $p$-block of $\mathbb{F}_p\mathfrak{S}_n$. 
Using the $p$-modular decomposition matrix of 
$\mathfrak{S}_n$, which is available for all cases considered here,
and \cite[Cor.~3.4]{BE}, we may determine $\dim(\Lieppbl(n))$ in 
advance, and \cite[Thm. 3.1]{BE} also tells us the
projective indecomposable direct summands of $\Liep(n)/\Lieppbl(n)$, 
so that next to $\Lieppf(n)$ only the projective indecomposable 
direct summands of $\Lieppbl(n)$ have to be determined.

\medskip

(c)\,
To find an $\mathbb{F}_p$-basis of $\Liep(n)$ or $\Lieppbl(n)$ 
in the first place, let $\epsilon_n\in\mathbb{F}_p\mathfrak{S}_n$ 
be the centrally primitive idempotent belonging to the principal 
$p$-block of $\mathbb{F}_p\mathfrak{S}_n$; recall that $\epsilon_n$ 
can be computed from the ordinary character table of $\mathfrak{S}_n$.
We now use the observation in \ref{noth:lieproperties}, saying that
an $\mathbb{F}_p$-basis of $\Liep(n)$ is given as
$$ \{\pi\cdot c_n\cdot\omega_n\mid \pi\in\mathfrak{S}_{n-1}\} 
   \subseteq \Liep(n)\subseteq \mathbb{F}_p\mathfrak{S}_n\,,$$
implying that an $\mathbb{F}_p$-spanning set of $\Lieppbl(n)$ is given as
$$ \{\pi\cdot c_n\cdot\omega_n\cdot\epsilon_n \mid 
     \pi\in\mathfrak{S}_{n-1}\} \subseteq \Liep(n)\cdot\epsilon_n
   =\Lieppbl(n) \subseteq \mathbb{F}_p\mathfrak{S}_n \,.$$

Hence our starting point is the regular representation 
$\mathbb{F}_p\mathfrak{S}_n$ of $\mathfrak{S}_n$, 
being equipped with its natural $\mathbb{F}_p$-basis.
Determining the permutation action of elements 
of $\mathfrak{S}_n$ on coordinate vectors with respect to this basis
essentially amounts to computing with permutations in $\mathfrak{S}_n$. 
This allows us to apply successively all elements of $\mathfrak{S}_{n-1}$
to $c_n\cdot\omega_n\in\mathbb{F}_p\mathfrak{S}_n$
and $c_n\cdot\omega_n\cdot\epsilon_n\in\mathbb{F}_p\mathfrak{S}_n$,
respectively;
to do this efficiently, we first find a Schreier tree of
$\mathfrak{S}_{n-1}$ in terms of some generating set, 
our favourite one being $\{(1,\ldots,n-1),(1,2)\}$.

Thus having found an $\mathbb{F}_p$-basis of $\Liep(n)$,
we directly determine the action of a generating set of $\mathfrak{S}_n$,
our favourite one again being $\{(1,\ldots,n),(1,2)\}$.
For $\Lieppbl(n)$, before doing so, we pick an $\mathbb{F}_p$-basis out of
the $\mathbb{F}_p$-spanning set obtained.
These tasks are efficiently solved using the linear algebra routines 
available in the {\sf MeatAxe}. 

\medskip

(d)\,
Next we proceed to find the projective indecomposable
direct summands of $\Lieppbl(n)$. To do so,
we apply a technique based on the considerations in \cite{LuxMueRin}.
In order to describe this we first recall the relevant notions:

Let $A$ be a finite-dimensional $K$-algebra, where $K$ is any field,
and let $S$ be a simple $A$-module.
Then, the endomorphism algebra $\End_A(S)$ is a skew field, and 
for $a\in A$ letting $\ker(a_S)$ denote the kernel of the $K$-endomorphism 
of $S$ induced by the action of $a$, we have 
$\dim(\End_A(S))\mid\dim(\ker(a_S))$.
Now $a\in A$ is called an \emph{$S$-peakword}, if 
$\dim(\ker(a_S^2))=\dim(\End_A(S))$, and
$\ker(a_T)=\{0\}$ for all simple $A$-modules $T$ not isomorphic to $S$.
In particular, if $K$ is a splitting field of $A$, then the first 
condition just becomes $\dim(\ker(a_S^2))=\dim(\ker(a_S))=1$.
In practice, peakwords are found by a random search, yielding a
Monte Carlo method, which for the case of $K$ being a (small) 
finite field is available in the {\sf MeatAxe}.

Let now $M$ be an $A$-module, and let $a\in A$ be an $S$-peakword.
Then, by \cite[Thm.2.5]{LuxMueRin}, the set of all 
submodules $L$ of $M$ such that $L/\Rad(L)\cong S$
concides with the set of all cyclic submodules of $M$ generated 
by some $v\in\bigcup_{i\geq 1}\ker(a_M^i)\smallsetminus\{0\}$.
Thus, in particular, all submodules of $M$ isomorphic to 
the projective cover $P_S$ of $S$ are found this way,
and for a cyclic submodule $L$ as above 
we have $L\cong P_S$ if and only if $\dim(L)=\dim(P_S)$.
Thus, if $A$ is a self-injective algebra, 
a random search through $\bigcup_{i\geq 1}\ker(a_M^i)$
yields a Monte Carlo method to find a largest direct summand
of $M$ being the direct sum of copies of $P_S$.
Note that $\dim(P_S)$ is indeed known in advance in all explict cases 
considered here, and that, if $K$ is a (small) finite field, then
techniques to compute cyclic submodules are available in the {\sf MeatAxe}.

\medskip

(e)\,
Thus, quotienting out the projective direct summands found,
we now have $\Lieppf(n)$ in our hands, at least with high probability.
In all cases considered here this module turns out to be small enough 
to apply to it the general techniques available in the {\sf MeatAxe}
to find direct sum decompositions. 
The latter techniques would also find a projective direct summand left over,
thus providing a verification of the above Monte Carlo results.
Actually, for the examples to be discussed below, this even shows that
$\Lieppf(n)$ is indecomposable and non-projective. 

Hence we may now assume that we have got 
a non-projective indecomposable $\mathbb{F}_p\mathfrak{S}_n$-module $M$,
for which we have to find a vertex and a source.
In order to do so,
we consider the restriction $\res_{P_n}^{\mathfrak{S}_n}(M)$ 
of $M$ to the Sylow $p$-subgroup $P_n$ of $\mathfrak{S}_n$;
recall from \ref{noth:vertex} that, 
since $M$ is relatively $P_n$-projective,
$\res_{P_n}^{\mathfrak{S}_n}(M)$ has an indecomposable direct
summand sharing a vertex and a source with $M$.

Hence we may assume that $M$ is an $\mathbb{F}_p P$-module, 
where $P$ is a $p$-group. 
Again, we have to find direct sum decompositions,
which can be speeded up by detecting particular
indecomposable direct summands beforehand.
Namely, in case of an $\mathbb{F}_p P$-module 
the strategy described in part~(d) specializes to the following:
the set of all submodules $L$ of $M$ such that $L/\Rad(L)$ is simple
is precisely the set of all non-zero cyclic submodules of $M$.
(Note that, in terms of the language used above,
since the trivial module is the only simple $\mathbb{F}_p P$-module,
the zero element in $\mathbb{F}_p P$ is a peakword.)
This leads to a straightforward Monte Carlo method to find a 
largest direct summand of $M$ that is the direct sum of copies of 
the regular module $\mathbb{F}_p P$; see \cite[Sect. 3.2]{DKZ}.

Quotienting
out projective direct summands 
we again, in all cases considered here, end up with a
module whose direct sum decomposition can be computed
with the general techniques available in the {\sf MeatAxe}.

Hence we may finally assume that $M$ is an
indecomposable non-projective  $\mathbb{F}_p P$-module 
such that we are in a position to use the techniques described 
in \cite[Sect. 3.1]{DKZ}, whose basic ingredient is Higman's 
Criterion for relative projectivity. An implementation is
available in {\sf MAGMA}, which, in particular, employs its 
facilities to compute with finite $p$-groups,
for example to determine subgroup lattices,
and sets of subgroup coset representatives. 

\medskip
(f)\,
To make sure that computational results are still valid when
going over to the algebraically closed field $F$ again, 
we always check that the indecomposable modules found 
are actually absolutely indecomposable; techniques to 
achieve that are available in the {\sf MeatAxe}.
Recall that it is well known that $\mathbb{F}_p$ is a 
splitting field of $\mathbb{F}_p\mathfrak{S}_n$, hence
absolute indecomposability is automatic anyway for the 
simple modules and the projective indecomposable modules found.

\end{noth}

\begin{noth}\label{noth:lie4}
{\bf Examining $\Liet(4)$.}\,
Let $p=2$. We examine the $\mathbb{F}_2\mathfrak{S}_4$-module $\Liet(4)$.

\medskip

(a)\,
Recall from Remark \ref{noth:lieproperties}(a) that $\dim(\Liet(4))=3!=6$. 
A dimension consideration shows that $\Liet(4)$ cannot
possibly contain a projective direct summand, hence $\Liet(4)$ coincides
with its projective-free part $\Liepft(4)$. Moreover, it is easily checked computationally,
that $\Liepft(4)$ is absolutely indecomposable, namely
$$ \Liepft(4)\cong \Omega^{-1}(D^{(3,1)}) \,,$$
where $D^{(3,1)}\cong\Inf_{\mathfrak{S}_3}^{\mathfrak{S}_4}(D^{(2,1)})$
is the simple $\mathbb{F}_2\mathfrak{S}_4$-module 
of dimension $2$, and the inflation is along the natural map
$\mathfrak{S}_4/E_4\cong\mathfrak{S}_3$.
Since $D^{(2,1)}$ is a projective simple $\mathbb{F}_2\mathfrak{S}_3$-module, 
$D^{(3,1)}$ is a trivial-source module with vertex $E_4$.
Thus we conclude that $\Liepft(4)$ has vertex $E_4$, and 
$\Omega^{-1}(\mathbb{F}_2)=\mathbb{F}_2E_4/\Soc(\mathbb{F}_2E_4)$ 
is an $E_4$-source, having dimension $3$.

Note that $\Omega^{-1}(\mathbb{F}_2)$ is an endo-permutation
module. Thus from Theorem \ref{thm:Dade}, 
and the remarks in \ref{noth:endopperm} and \ref{noth:sylow}(b), 
we conclude that $\Liepft(4)$ is an endo-$p$-permutation module.

\medskip

(b)\,
In view of the above and the subsequent results, it seems worthwhile to 
show that $\Liepft(4)$ is \emph{not} an endo-permutation
$\mathbb{F}_2\mathfrak{S}_4$-module. To this end, we 
compute an explicit indecomposable direct sum decomposition of 
$\Liepft(4)\otimes\Liepft(4)^*$:
\begin{equation}\label{equ:Lie4tens}
 \Liepft(4)\otimes \Liepft(4)^* \cong
D^{(3,1)} \oplus \ind_{\mathfrak{A}_4}^{\mathfrak{S}_4}(\mathbb{F}_2)
\oplus P^{(4)} \oplus 3\cdot P^{(3,1)} \,, 
\end{equation}
where both projective indecomposable $\mathbb{F}_2\mathfrak{S}_4$-modules
$P^{(4)}$ and $P^{(3,1)}$ have dimension $8$.
Note that, in accordance with part~(a), we indeed observe that all 
indecomposable direct summands of
$\Liepft(4)\otimes\Liepft(4)^*$ are trivial-source modules.

To show that $\Liepft(4)\otimes\Liepft(4)^*$ is not a permutation 
$\mathbb{F}_2\mathfrak{S}_4$-module, assume to the contrary that it is.
Then there is some $H\leq\mathfrak{S}_4$ such that $D^{(3,1)}$ is 
isomorphic to a direct summand of $\ind_H^{\mathfrak{S}_4}(\mathbb{F}_2)$ 
and such that $\ind_H^{\mathfrak{S}_4}(\mathbb{F}_2)$ is isomorphic 
to a direct summand of $\Liepft(4)\otimes \Liepft(4)^*$. 
In particular, $D^{(3,1)}$ is then relatively $H$-projective,
and since, by (a), $E_4\unlhd \mathfrak{S}_4$ is a vertex of $D^{(3,1)}$,
we infer $E_4\leq H$.
On the other hand, $H$ cannot possibly contain a Sylow $2$-subgroup of 
$\mathfrak{S}_4$, since otherwise $\ind_H^{\mathfrak{S}_4}(\mathbb{F}_2)$
has the trivial $\mathbb{F}_2\mathfrak{S}_4$-module as a direct summand.
This leaves the cases $H\in\{E_4,\mathfrak{A}_4\}$. 
But if $H=\mathfrak{A}_4$ then
$D^{(3,1)}\nmid \ind_H^{\mathfrak{S}_4}(\mathbb{F}_2)$, and if $H=E_4$ then
$2\cdot D^{(3,1)}\mid \ind_H^{\mathfrak{S}_4}(\mathbb{F}_2)$. 
In any case, we obtain a contradiction.
%
\end{noth}

\begin{noth}\label{noth:lie8}
{\bf Examining $\Liet(8)$.}\,
Let $p=2$. We examine the $\mathbb{F}_2\mathfrak{S}_8$-module $\Liet(8)$.

\medskip

(a)\,
Recall from Remark \ref{noth:lieproperties}(a) that $\dim(\Liet(8))=7!=5040$.
Moreover, using the $2$-modular decomposition matrix of $\mathfrak{S}_8$
and \cite[Cor.~3.4]{BE}, we find $\dim(\Liepblt(8))=4016$. 
By work of Selick--Wu \cite{SW}, it is known that
\begin{equation}\label{equ:Lie8dec}
\Liepblt(8)\cong \Liepft(8)\oplus 2\cdot P^{(6,2)}
                 \oplus P^{(5,3)}\oplus 4\cdot P^{(4,3,1)} \,,
\end{equation}
so that we infer that $\dim(\Liepft(8))=816=2^4\cdot 3\cdot 17$. 
We have verified the decomposition (\ref{equ:Lie8dec}) independently,
with the computational techniques described in \ref{noth:computations}.
In addition to the calculations in \cite{SW}, 
we have checked explicitly that $\Liepft(8)$ is actually 
absolutely indecomposable.

\medskip

(b)\,
We will subsequently describe the vertices and sources of 
the projective-free part $\Liepft(8)$.
In order to do so, we consider the restriction of $\Liepft(8)$
to the Sylow $2$-subgroup $P_8$ of $\mathfrak{S}_8$; 
note that, since $|P_8|=2^7$, from \ref{noth:vertex} we conclude
that every vertex of $\Liepft(8)$ has order at least $8$. 
Our computations yield the following decomposition:
\begin{equation*}\label{equ:Lie8P8}
\res_{P_8}^{\mathfrak{S}_8}(\Liepft(8))=M_1\oplus M_2\oplus \mathrm{(cyc)},
\end{equation*} 
where `$\mathrm{(cyc)}$' denotes a direct sum of absolutely
indecomposable $\mathbb{F}_2P_8$-modules with vertex $Z(P_8)$ of order 2,
and with trivial sources.

The direct summand $M_2$ is absolutely indecomposable of dimension 96,
and has vertex 
$$V:=\langle (1,3)(2,4)(5,6)(7,8),(1,4)(2,3)(5,8)(6,7)\rangle
  \cong C_2\times C_2 $$
of order $4$, and a $V$-source isomorphic to 
$\mathbb{F}_2V/\Soc(\mathbb{F}_2V)\cong\Omega^{-1}(\mathbb{F}_2)$. 
In particular, the sources of $M_2$ are endo-permutation modules.

The remaining direct summand, $M_1$, is absolutely indecomposable of dimension 336,
and has vertex $E_8$ and an $E_8$-source $S$ of dimension 21 satisfying
$$\End_{\mathbb{F}_2}(S)\cong 
S\otimes S^*\cong \mathbb{F}_2\oplus\bigoplus_{Q<E_8,\,|Q|=2} 
2\cdot \ind_Q^{E_8}(\mathbb{F}_2)\oplus\mathrm{(proj)},$$
where `$\mathrm{(proj)}$' denotes a projective $\mathbb{F}_2E_8$-module. 
Consequently, $\End_{\mathbb{F}_2}(S)$ is a permutation 
$\mathbb{F}_2E_8$-module, that is, $S$ is an endo-permutation
$\mathbb{F}_2E_8$-module. In fact, by Theorem \ref{thm:Dade}, 
the isomorphism type of $S$ is determined by the following isomorphism,
which is easily verified computationally, using the techniques in \ref{noth:computations}:
$$\Omega^3(\mathbb{F}_2)\otimes\bigotimes_{Q<E_8,\,|Q|=2}
  \Inf_{E_8/Q}^{E_8}(\Omega^{-1}((\mathbb{F}_2)_{E_8/Q}))\cong 
S\oplus\mathrm{(proj)}.$$
Note that $S$ is the only non-projective direct summand occurring.

In conclusion, this shows that $\Liepft(8)$ has vertex $E_8$ and
endo-permutation source $S$. 
In particular, by \ref{noth:endopperm} and \ref{noth:sylow}(b),
we conclude that $\Liepft(8)$ is an endo-$p$-permutation module.
\end{noth}

\begin{noth}\label{noth:lie9}
{\bf Examining $\Lieth(9)$.}\,
Next let $p=3$. We examine the $\mathbb{F}_3\mathfrak{S}_9$-module 
$\Lieth(9)$.

\medskip

(a)\,
Recall from Remark \ref{noth:lieproperties}(a) that 
$\dim(\Lieth(9))=8!=40320$.
Moreover, using the $3$-modular decomposition matrix of $\mathfrak{S}_9$
and \cite[Cor.~3.4]{BE}, we find $\dim(\Liepblth(9))=16020$. 
Employing the techniques described in \ref{noth:computations},
we obtain the following decomposition
\begin{align*}
\Liepblth(9)&\cong 
2\cdot P^{(7,1^2)}\oplus 
5\cdot P^{(6,3)}  \oplus 
3\cdot P^{(6,2,1)}\oplus 
4\cdot P^{(5,2^2)}\oplus 
2\cdot P^{(4,3,2)}\oplus 
       P^{(4^2,1)}\oplus 
4\cdot P^{(3^2,2,1)} \\ 
 &\oplus\Liepfth(9),
\end{align*}
where hence $\Liepfth(9)$ has dimension
$1683=3^2\cdot 11\cdot 17$, and turns out to be absolutely indecomposable.

\medskip

(b)\, 
To describe the vertices and sources of the projective-free part 
$\Liepfth(9)$, we first note that from $|P_9|=3^4$ and \ref{noth:vertex} 
we conclude that every vertex of $\Liepfth(9)$ has order at least $9$. 
We determine an indecomposable direct sum  decomposition
of the restriction of $\Liepfth(9)$ to $P_9$, and get
$$\res_{P_9}^{\mathfrak{S}_9}(\Liepfth(9))\cong 
  N_1\oplus 2\cdot N_2\oplus 4\cdot N_3\oplus\mathrm{(proj)},$$
where $N_2\not\cong N_3$ are absolutely indecomposable of dimension 54 each,
having non-conjugate cyclic vertices of order 3, and 
endo-permutation sources of dimension 2.

The direct summand $N_1$ is absolutely indecomposable of dimension 144,
and has vertex $E_9$ and an $E_9$-source $S'$ of dimension 16 satisfying
$$\End_{\mathbb{F}_3}(S')\cong S'\otimes (S')^*\cong 
  \mathbb{F}_3\oplus \bigoplus_{Q<E_9,\,|Q|=3}\ind_Q^{E_9}(\mathbb{F}_3)
  \oplus\mathrm{(proj)}.$$
Consequently, $\End_{\mathbb{F}_3}(S')$ is a permutation
$\mathbb{F}_3E_9$-module, that is, $S'$ is an endo-permutation
$\mathbb{F}_3E_9$-module. Its isomorphism type, in the
sense of Theorem \ref{thm:Dade}, is determined by the following isomorphism,
where again $S'$ is the only non-projective direct summand occurring:
$$\Omega^{-2}(\mathbb{F}_3)\otimes\bigotimes_{Q<E_9,\,|Q|=3} 
  \Inf_{E_9/Q}^{E_9}(\Omega((\mathbb{F}_3)_{E_9/Q}))
  \cong S'\oplus\mathrm{(proj)}.$$
Also this decomposition is verified computationally 
via the techniques described in \ref{noth:computations}.

In conclusion, this shows that $\Liepfth(9)$ has vertex $E_9$ and
endo-permutation source $S'$. 
In particular, by \ref{noth:endopperm} and \ref{noth:sylow}(b),
we conclude that $\Liepfth(9)$ is an endo-$p$-permutation module.
\end{noth}

\section{A Reduction Theorem}\label{sec:reduction}

The aim of this section is to establish Theorem~\ref{thm:vertexLie},
which will allow for a partial reduction of the question concerning
vertices and sources of indecomposable direct summands
of $\mathrm{Lie}_p(n)$ to the case where $n$ is a $p$-power.
The key ingredients will be Theorem~\ref{prop:BE},
and the results in \cite{Ku} on vertices
of indecomposable modules of wreath products.
Therefore, we start out by collecting a number of facts on wreath
products and their representations, which we will then apply
in the context of Lie modules.

\begin{noth}\label{noth:wreathmodules}
{\bf Wreath products and their modules.}\,
(a)\, 
Let $G$ be a finite group, and consider the wreath product
$$G\wr \mathfrak{S}_n:=\{(g_1,\ldots,g_n; \sigma)
  \mid g_1,\ldots,g_n\in G,\, \sigma\in\mathfrak{S}_n\}.$$
Recall that the multiplication in $G\wr \mathfrak{S}_n$ is given by
\begin{equation}\label{equ:wreath multi}
(g_1,\ldots,g_n;\sigma)(h_1,\ldots,h_n;\pi)
=(g_1h_{\sigma^{-1}(1)},\ldots, g_nh_{\sigma^{-1}(n)};\sigma\pi)\,,
\end{equation}
for $g_1,\ldots,g_n,h_1,\ldots,h_n\in G$ and $\sigma,\pi\in\mathfrak{S}_n$.
Hence we have the natural epimorphism
$$\overline{\rule{0em}{0.5em}\rule{0.5em}{0em}}: 
G\wr \mathfrak{S}_n\longrightarrow \mathfrak{S}_n,\; 
(g_1,\ldots,g_n;\sigma)\longmapsto\sigma\,.$$
We denote by $G^n$ the {\it base group} of $G\wr\mathfrak{S}_n$, that is,  
$$G^n=\{(g_1,\ldots,g_n;1)\mid g_1,\ldots, g_n\in G\}
\unlhd G\wr\mathfrak{S}_n .$$ 
Moreover, letting $\sigma^\sharp:=(1,\ldots,1;\sigma)\in G\wr\mathfrak{S}_n$,
for $\sigma\in\mathfrak{S}_n$, we get an isomorphism
$$\mathfrak{S}_n^\sharp:=\{\sigma^\sharp\mid 
   \sigma\in\mathfrak{S}_n\}\cong \mathfrak{S}_n ;$$
note that the map $(-)^\sharp: \mathfrak{S}_n\longrightarrow G\wr\mathfrak{S}_n$ 
is a section for the natural epimorphism 
$\overline{\rule{0em}{0.5em}\rule{0.5em}{0em}}: 
G\wr \mathfrak{S}_n\longrightarrow \mathfrak{S}_n$.
More generally, if $H\leq G$ and $U\leq \mathfrak{S}_n$ then
we further set 
$U^\sharp:=\{\sigma^\sharp\mid \sigma\in U\}\leq \mathfrak{S}_n^\sharp$, 
as well as 
$H^n:=\{(g_1,\ldots,g_n;1)\mid g_1,\ldots,g_n\in H\}\leq G^n$, 
and 
$$ H\wr U:=\{(g_1,\ldots,g_n;\sigma)\mid g_1,\ldots,g_n\in H,\, \sigma\in U\}
 \leq G\wr\mathfrak{S}_n .$$

\medskip
(b)\, Let $M$ be an $FG$-module.
Then the (outer) tensor product 
$M^{\otimes n}=M\otimes\cdots\otimes M$ 
becomes an $F[G\wr\mathfrak{S}_n]$-module via
$$(g_1,\ldots,g_n;\sigma)(x_1\otimes\cdots\otimes x_n):=
g_1x_{\sigma^{-1}(1)}\otimes\cdots\otimes g_nx_{\sigma^{-1}(n)},$$
for $g_1,\ldots,g_n\in G$, $\sigma\in\mathfrak{S}_n$, and 
$x_1,\ldots,x_n\in M$.
This module is called a {\it tensor-induced} module.

\mbox{}From now on, 
we denote by $\Lambda(m,n)$ the set of compositions of $n$
with at most $m$ non-zero parts.
If $\lambda=(\lambda_1,\ldots,\lambda_m)\in\Lambda(m,n)$ 
then we
denote by $\mathfrak{S}_\lambda$ the corresponding (standard) Young subgroup 
$\mathfrak{S}_{\lambda_1}\times\cdots\times \mathfrak{S}_{\lambda_m}$
of $\mathfrak{S}_n$. With this notation, 
$$ (G\wr\mathfrak{S}_{\lambda_1})\times\cdots
\times (G\wr\mathfrak{S}_{\lambda_m})\cong G\wr\mathfrak{S}_\lambda\leq G\wr\mathfrak{S}_n .$$
Thus, if $M_1,\ldots,M_m$ are $FG$-modules, the (outer) tensor product 
\begin{equation*}
M^{\otimes\lambda}:=
M_1^{\otimes\lambda_1}\otimes\cdots\otimes M_m^{\otimes\lambda_m}
\end{equation*}
 carries
an $F[G\wr\mathfrak{S}_\lambda]$-module structure.

Moreover, suppose again that $H\leq G$ and $U\leq\mathfrak{S}_n$, and let $L$ be an $FU$-module.
Then, via the map $(-)^\sharp$, the $FU$-module $L$ can
be viewed as an $FU^\sharp$-module, 
which we denote by $L^\sharp$.
Via inflation along the natural epimorphism 
$\overline{\rule{0em}{0.5em}\rule{0.5em}{0em}}$,
the $FU$-module $L$ becomes also an 
$F[H\wr U]$-module, which we denote by 
$\widehat{L}:=\Inf^{H\wr U}_{U}(L)$.
Thus we have 
$\res^{H\wr U}_{U^\sharp}(\widehat{L}) 
 = L^\sharp$.

\medskip
(c)\, Let $N$ be an $F\mathfrak{S}_n$-module, and again let $M$ be an $FG$-module.
In this section, we will describe vertices 
and sources of indecomposable direct summands of the 
$F[G\wr\mathfrak{S}_n]$-module $M^{\otimes n}\otimes \hatN$ in terms of those
of the indecomposable direct summands of $M$ and $N$. We, therefore,
recall the structure of the indecomposable direct summands of the 
$F[G\wr\mathfrak{S}_n]$-modules
$M^{\otimes n}$ and $\hatN$, respectively:

Let $N_1,\ldots,N_l$ be pairwise non-isomorphic indecomposable
$F\mathfrak{S}_n$-modules, and $b_1,\ldots,b_l\in\mathbb{N}$ be such that
$N\cong b_1N_1\oplus\cdots\oplus b_lN_l$. Then we get
\begin{equation*}
\hatN=\Inf_{\mathfrak{S}_n}^{G\wr\mathfrak{S}_n}(N)\cong 
\bigoplus_{i=1}^l b_i\Inf_{\mathfrak{S}_n}^{G\wr\mathfrak{S}_n}(N_i) 
=\bigoplus_{i=1}^l b_i\hatN_i,
\end{equation*}
where the $F[G\wr\mathfrak{S}_n]$-modules 
$\hatN_i:=\Inf_{\mathfrak{S}_n}^{G\wr\mathfrak{S}_n}(N_i)$
are pairwise non-isomorphic and indecomposable.
Thus, the indecomposable direct summmands of the $F\mathfrak{S}_n$-module 
$N$ and those of the $F[G\wr\mathfrak{S}_n]$-module $\hatN$ are in natural
bijection, and hence in the sequel we may assume that $N$ is
indecomposable.

As for $M^{\otimes n}$, 
let $M_1,\ldots,M_m$ be pairwise non-isomorphic indecomposable 
$FG$-modules, and let $a_1,\ldots,a_m\in\mathbb{N}$ be such that 
\begin{equation*}
M\cong a_1M_1\oplus\cdots\oplus a_mM_m\,.
\end{equation*}
Then we have the following well-known result;
we include a proof for the readers' convenience.
\end{noth}

\begin{lemma}\label{lemma:directsum}
With the notation as in \ref{noth:wreathmodules}(c),
\begin{equation*}\label{equ:decX}
M^{\otimes n}\cong\bigoplus_{\lambda\in\Lambda(m,n)}c_\lambda\cdot 
\ind_{G\wr\mathfrak{S}_{\lambda}}^{G\wr\mathfrak{S}_n}
(M_1^{\otimes \lambda_1}\otimes\cdots\otimes M_m^{\otimes \lambda_m})
\end{equation*}
is an indecomposable direct sum decomposition of the 
$F[G\wr\mathfrak{S}_n]$-module
$M^{\otimes n}$,
for suitable $c_\lambda\in\mathbb{N}$. 
\end{lemma}

\begin{proof}
We have an isomorphism of $F[G\wr \mathfrak{S}_n]$-modules
\begin{equation*}\label{equ:M}
M^{\otimes n}\cong 
\bigoplus_{\lambda=(\lambda_1,\ldots,\lambda_m)\in\Lambda(m,n)}
c_\lambda\cdot(\bigoplus \tildeM_1\otimes\cdots \otimes \tildeM_n),
\end{equation*}
the inner sum being taken over all $n$-tuples $(\tildeM_1,\ldots,\tildeM_n)$
of $FG$-modules satisfying 
$$|\{1\leq j\leq n\mid \tildeM_j=M_i\}|=\lambda_i, 
\quad\text{for}\quad i=1,\ldots,m .$$
The respective coefficient $c_\lambda$ equals 
$a_1^{\lambda_1}\cdots a_m^{\lambda_m}$.

Given $\lambda=(\lambda_1,\ldots,\lambda_m)\in\Lambda(m,n)$, the sum
$\bigoplus (\tildeM_1\otimes\cdots\otimes \tildeM_n)$ is a transitive 
imprimitive $F[G\wr\mathfrak{S}_n]$-module, and
the direct summands $\tildeM_1\otimes\cdots\otimes \tildeM_n$ 
form a system of imprimitivity.
One of these direct summands equals 
$M^{\otimes\lambda}=
M_1^{\otimes\lambda_1}\otimes\cdots\otimes M_m^{\otimes\lambda_m}$.
Its restriction to the base group $G^n$ of $G\wr\mathfrak{S}_n$ 
is indecomposable, and
its inertial group in $G\wr\mathfrak{S}_n$ equals $G\wr\mathfrak{S}_\lambda$.
Thus, by \cite[50.2]{CR0}, we deduce that 
$$\bigoplus \tildeM_1\otimes\cdots\otimes \tildeM_n\cong 
\ind_{G\wr\mathfrak{S}_\lambda}^{G\wr\mathfrak{S}_n}(M^{\otimes\lambda})$$
and, by \cite[Prop. 4.1]{Ku},
$\ind_{G\wr\mathfrak{S}_\lambda}^{G\wr\mathfrak{S}_n}(M^{\otimes\lambda})$
is an indecomposable $F[G\wr\mathfrak{S}_n]$-module.
\end{proof}

\begin{noth}\label{noth:L}
{\bf Wreath products and vertices.}\,
We retain the notation from \ref{noth:wreathmodules}(c).
In particular, we suppose that $N$ is an indecomposable $F\mathfrak{S}_n$-module.
We now want to examine the vertices and sources of the 
indecomposable direct summands of the $F[G\wr \mathfrak{S}_n]$-module
$M^{\otimes n}\otimes \widehat{N}$; the answer is given in Theorem~\ref{thm:wreathreduce} below.
To this end, let $P$ be a Sylow $p$-subgroup of $G$, and 
for $j=1,\ldots,m$ 
let $R_j$ be a vertex of the $FG$-module $M_j$. 

\medskip
(a)\, Now let $L$ be an indecomposable direct summand of
$M^{\otimes n}\otimes \widehat{N}$.
Then, by \ref{noth:wreathmodules}(c) and Lemma \ref{lemma:directsum}, there
is some $\lambda=(\lambda_1,\ldots,\lambda_m)\in\Lambda(m,n)$ 
such that $L$ is isomorphic to a direct summand of 
\begin{align*}
\ind_{G\wr\mathfrak{S}_\lambda}^{G\wr\mathfrak{S}_n}
(M^{\otimes\lambda})\otimes\hatN
&\cong \ind_{G\wr\mathfrak{S}_\lambda}^{G\wr\mathfrak{S}_n}
(M^{\otimes\lambda}
\otimes\res^{G\wr\mathfrak{S}_n}_{G\wr\mathfrak{S}_\lambda}(\hatN)) \\
&\cong 
\ind_{G\wr\mathfrak{S}_\lambda}^{G\wr\mathfrak{S}_n}
(M^{\otimes\lambda}\otimes 
\Inf_{\mathfrak{S}_\lambda}^{G\wr\mathfrak{S}_\lambda}
(\res^{\mathfrak{S}_n}_{\mathfrak{S}_\lambda}(N)))\,.
\end{align*}
Then, by \cite[Prop. 5.1]{Ku} and the discussion preceding it,
there is an indecomposable direct summand $N'$ of
$\res^{\mathfrak{S}_n}_{\mathfrak{S}_\lambda}(N)$ such that 
\begin{equation*}\label{eqn L'}
L\cong \ind_{G\wr\mathfrak{S}_\lambda}^{G\wr\mathfrak{S}_n}(L'),
\quad\text{where}\quad
L':=M^{\otimes\lambda}\otimes\hatN'\,.
\end{equation*}
In particular,
$L'$ is an indecomposable $F[G\wr\mathfrak{S}_\lambda]$-module.
Now, if $Q'\leq\mathfrak{S}_\lambda$ 
is a vertex of $N'$, 
then
\begin{equation}\label{eqn Q}
Q:=(R_1^{\lambda_1}\times\cdots\times R_m^{\lambda_m})\rtimes (Q')^\sharp
\leq G\wr\mathfrak{S}_\lambda\leq G\wr\mathfrak{S}_n
\end{equation} 
is a common vertex of $L$ and $L'$.

\medskip
(b)\, We consider a common $Q$-source $S$
of the $F[G\wr\mathfrak{S}_n]$-module $L$
and the $F[G\wr\mathfrak{S}_\lambda]$-module $L'$.
To this end, we from now on additionally suppose that
each of the $FG$-modules $M_1,\ldots,M_m$ has trivial sources.
Note that this, in particular, includes the case that all these
modules are projective.

Let $P_\lambda=P_{\lambda_1}\times\cdots\times P_{\lambda_m}$ 
be a Sylow $p$-subgroup of the Young subgroup
$\mathfrak{S}_\lambda=\mathfrak{S}_{\lambda_1}\times\cdots\times 
 \mathfrak{S}_{\lambda_m}$.
Then, in consequence of \cite[Prop. 1.2, Prop. 3.1]{Ku},
\begin{equation*}\label{eqn R}
R_\lambda:=(R_1^{\lambda_1}\times\cdots\times R_m^{\lambda_m})\rtimes 
           P_\lambda^\sharp
=(R_1\wr P_{\lambda_1})\times\cdots\times 
 (R_m\wr P_{\lambda_m})\leq 
G\wr\mathfrak{S}_\lambda 
\end{equation*}
is  a vertex of the indecomposable $F[G\wr \mathfrak{S}_\lambda]$-module
$M^{\otimes\lambda}$, and 
$M^{\otimes\lambda}$ is a trivial-source module, that is, 
$M^{\otimes\lambda}\mid \ind_{R_\lambda}^{G\wr \mathfrak{S}_\lambda}(F)$.
(Note that the assertion on vertices is just a special case of (\ref{eqn Q}).)

\smallskip

Suppose that $S'$ is  a $Q'$-source of the 
$F\mathfrak{S}_\lambda$-module $N'$. 
\mbox{}From \cite[Prop. 2]{Har} we deduce that the 
$F[G\wr\mathfrak{S}_\lambda]$-module
$\hatN':=\Inf_{\mathfrak{S}_\lambda}^{G\wr\mathfrak{S}_\lambda}(N')$
has vertex $P\wr Q'$, and
$\widehat{S'}:=\Inf_{Q'}^{P\wr Q'}(S')$ 
is a $(P\wr Q')$-source of $\hatN'$. Thus we have  
$\hatN'\mid \ind_{P\wr Q'}^{G\wr\mathfrak{S}_\lambda} (\widehat{S'})$.
Hence Mackey's Tensor Product Theorem 
shows that there is some $g\in G\wr\mathfrak{S}_\lambda$ such that
$L'$ is a direct summand of
$\ind_{\widetilde{Q}}^{G\wr\mathfrak{S}_\lambda}
  (\res^{{}^g(P\wr Q')}_{\widetilde{Q}}({}^g\widehat{S'}))$,
where 
$$ \widetilde{Q}:= R_\lambda\cap {}^g(P\wr Q')
   \leq G\wr\mathfrak{S}_\lambda\, .$$

Hence $S$ is a direct summand of 
$\res^{G\wr\mathfrak{S}_\lambda}_{Q}(
\ind_{\widetilde{Q}}^{G\wr\mathfrak{S}_\lambda}
  (\res^{{}^g(P\wr Q')}_{\widetilde{Q}}({}^g\widehat{S'})))$,
thus, by Mackey's Theorem, there is some $h\in G\wr\mathfrak{S}_\lambda$
such that $S$ is a direct summand of
$$ \ind_{Q\cap{}^h \widetilde{Q}}^Q(
   \res^{{}^h\widetilde{Q}}_{Q\cap{}^h \widetilde{Q}}{}^h(
   \res^{{}^g(P\wr Q')}_{\widetilde{Q}}
   ({}^g\widehat{S'})))
=\ind_{Q\cap{}^h \widetilde{Q}}^Q(
 \res^{{}^{hg}(P\wr Q')}_{Q\cap{}^h \widetilde{Q}}({}^{hg}\widehat{S'})) .$$
Since $S$ has vertex $Q$, we infer $Q\cap{}^h \widetilde{Q}=Q$, so 
$S$ is a direct summand of $\res^{{}^{hg}(P\wr Q')}_Q({}^{hg}\widehat{S'})$.

Now we consider the natural epimorphism 
${}^-:G\wr\mathfrak{S}_\lambda\longrightarrow \mathfrak{S}_\lambda$, and let
$\sigma:=\overline{hg}\in \mathfrak{S}_\lambda$. Then
$$ Q'=\overline{(R_1^{\lambda_1}\times\cdots\times R_m^{\lambda_m})
   \rtimes (Q')^\sharp}
=\overline{Q}\leq\overline{{}^{hg}(P\wr Q')}={}^\sigma Q'\,. $$
Hence we have $\sigma\in N_{\mathfrak{S}_\lambda}(Q')$.
Moreover, since the base group $G^n$ acts trivially on 
$\widehat{S'}$, we infer that 
$Q\cap G^n=R_1^{\lambda_1}\times\cdots\times R_m^{\lambda_m}$
acts trivially on $S$, and since $S'$ is an indecomposable $FQ'$-module,
we finally conclude that 
\begin{equation*}\label{eqn S}
S\cong \res^{{}^{hg}(P\wr Q')}_Q({}^{hg}\widehat{S'})
\cong \Inf_{Q'}^Q({}^{\sigma}S')\,.
\end{equation*}
Recall from \ref{noth:vertex} that, since $S'$ is a $Q'$-source 
of $N'$, so is ${}^{\sigma}S'$.

\medskip
(c)\, Keep the notation as in part~(b), and suppose additionally that the
$FG$-module $M$ is projective, that is, $R_1=\cdots =R_m=\{1\}$, 
thus $Q=(Q')^\sharp$.
Furthermore, we now get 
$\sigma^\sharp=(1,\ldots,1;\sigma)
\in N_{G\wr\mathfrak{S}_\lambda}((Q')^\sharp)$, 
since $\sigma\in N_{\mathfrak{S}_\lambda}(Q')$. Hence we have 
$S\cong ({}^\sigma S')^\sharp={}^{\sigma^\sharp} ((S')^\sharp)$.
Since $S'$ was an arbitrary $Q'$-source of $N'$, this shows that 
indeed every $Q'$-source of $N'$, in the way just described, 
yields a common $Q$-source of $L$ and $L'$.
\end{noth}

Altogether we have, in particular, now proved the following:

\begin{thm}\label{thm:wreathreduce}
Let $M$ be an $FG$-module, let $N$ be an indecomposable 
$F\mathfrak{S}_n$-module, and let $L$ be an indecomposable
direct summand of the $F[G\wr\mathfrak{S}_n]$-module 
$M^{\otimes n}\otimes \hatN$. Suppose that
$M\cong\bigoplus_{j=1}^m a_jM_j$ is an indecomposable
direct sum decomposition of the $FG$-module $M$.
For $j=1,\ldots,m$, let $R_j$ be a vertex of $M_j$. 

\medskip

{\rm (a)}\, There are some 
$\lambda=(\lambda_1,\ldots,\lambda_m)\in\Lambda(m,n)$ and an 
indecomposable direct summand $N'$ of
$\res_{\mathfrak{S}_\lambda}^{\mathfrak{S}_n}(N)$ such that
$L\cong \ind_{G\wr\mathfrak{S}_\lambda}^{G\wr\mathfrak{S}_n}
 (M^{\otimes\lambda}\otimes \widehat{N'})$.
For every vertex $Q'\leq \mathfrak{S}_\lambda$ of $N'$, the group
$$Q:=(R_1^{\lambda_1}\times\cdots\times R_m^{\lambda_m})\rtimes(Q')^\sharp
  \leq G\wr\mathfrak{S}_\lambda\leq G\wr\mathfrak{S}_n$$
is a vertex of $L$.

\medskip

{\rm (b)}\, Suppose in addition that $M_1,\ldots,M_m$ are
trivial-source modules. Then there is, moreover, a $Q'$-source $S'$
of $N'$ such that $\Inf_{Q'}^Q(S')$ is a $Q$-source of $L$. 
Here the inflation is taken with respect to the natural 
epimorphism $Q\longrightarrow Q'$.

\medskip

{\rm (c)}\, If $M$ is a projective $FG$-module, and if $\lambda$ and $N'$
are as is part~(a), then $(Q')^\sharp$ is a vertex of $L$. If $S'$ is any 
$Q'$-source of $N'$ then $(S')^\sharp$ is also a $(Q')^\sharp$-source of $L$.
\end{thm}

The next theorem concerning decompositions of Lie modules
is a  consequence of the results in \cite{Br}, \cite{BS} and \cite{LT}.
This decomposition and Theorem~\ref{thm:wreathreduce} will then enable us
to reduce the problem of determining vertices of indecomposable
direct summands of $\LieFpf(n)$ to the case where $n$ is a $p$-power.

\begin{thm}\label{prop:BE}
Let $k>1$ with $p\nmid k$. Then, for every $s\geq 0$, there is a  projective 
$F\mathfrak{S}_{kp^{s}}$-module $X_{kp^s}$
such that, for all $d\geq 0$, one has
\begin{equation}\label{eqn lie dec}
\LieF(k p^d)\cong\bigoplus_{t=0}^d
 \ind_{\mathfrak{S}_{kp^{t}}\wr\mathfrak{S}_{p^{d-t}}}^{\mathfrak{S}_{kp^d}}
(X_{kp^t}^{\otimes p^{d-t}}\otimes\widehat{\LieF(p^{d-t})})\,.
\end{equation}
\end{thm}

\begin{proof}
We fix $n:=k p^d$, for some $d\geq 0$, and consider 
the natural $F[\GL_n(F)]$-module $F^n$.
Recall from \ref{noth:Lie}(b) that the Lie module $\LieF(n)$ 
is the image of the $F[\GL_n(F)]$-module $L^n(F^n)$ under the Schur functor
$\mathfrak{W}^{n}$,
taking $n$-homogeneous polynomial $F[\GL_n(F)]$-modules to 
$F\mathfrak{S}_n$-modules. 

By \cite[Thm.~3.4]{Br},
for all $t\geq 0$ there are idempotents 
$f_{kp^t}$ in the group algebra $F\mathfrak{S}_{kp^t}$, 
only depending on $p$, $k$ and $t$, but independent of $n$, 
such that there is an isomorphism of $F[\GL_n(F)]$-modules 
\begin{equation}\label{eqn BS thm}
L^n(F^n)\cong \bigoplus_{t=0}^d L^{p^{d-t}}((F^n)^{\otimes kp^t}*f_{kp^t})\,;
\end{equation}
recall from \ref{noth:E} that $\mathfrak{S}_{kp^t}$ acts from the right
on $(F^n)^{\otimes kp^t}$ by place permutations.

Now fix some $0\leq t\leq d$, and set $m:=kp^t$ and $q:=p^{d-t}$.
Suppose that $V$ is any $m$-homogeneous polynomial
$F[\GL_n(F)]$-module, so that the Schur functor can be applied to $V$,
yielding the left $F\mathfrak{S}_{m}$-module 
$\mathfrak{W}^{m}(V)$.
Then, by \cite[Cor. 3.2]{LT}, there is an isomorphism
of $F\mathfrak{S}_n$-modules
\begin{equation*}\label{eqn LT}
\mathfrak{W}^{n}(L^{q}(V))\cong 
\ind_{\mathfrak{S}_{m}\wr\mathfrak{S}_{q}}^{\mathfrak{S}_n}
(\mathfrak{W}^{m}(V)^{\otimes q}\otimes 
\widehat{\LieF(q)})\,.
\end{equation*}
We apply this to our fixed direct summand on the right-hand side of
(\ref{eqn BS thm}), with $V:=(F^n)^{\otimes m}*f_{m}$.
Thus letting
$$ X_m:=\mathfrak{W}^{m}((F^n)^{\otimes m}*f_{m}) $$
yields the decomposition (\ref{eqn lie dec}), and it remains to
show that $X_{m}$ is a projective $F\mathfrak{S}_{m}$-module,
and does not depend on the $F[\GL_n(F)]$-module $F^n$ used to define it.

But, since $n\geq m$, the Schur functor $\mathfrak{W}^{m}$ 
takes $(F^n)^{\otimes m}$ to the regular module $F\mathfrak{S}_{m}$, 
and the isomorphism $ \mathfrak{W}^{m}((F^n)^{\otimes m})\cong F\mathfrak{S}_m$ translates
the place permutation action into right multiplication. Thus we deduce
$$X_{m}= \mathfrak{W}^{m}((F^n)^{\otimes m}*f_{m})=\mathfrak{W}^{m}((F^n)^{\otimes m})*f_{m}
  \cong F\mathfrak{S}_{m}\cdot f_{m}\,,$$
which is of course a projective $F\mathfrak{S}_{m}$-module,
and independent of $n$.
\end{proof}


Note that also for the case $k=1$, which is excluded from the
present discussion, we have a decomposition similar to (\ref{eqn lie dec}), 
but in this case becoming trivial inasmuch as for $t=0$ we get the 
trivial $F\mathfrak{S}_1$-module $X_{1}\cong F$, 
and $X_{p^t}=\{0\}$ for $t\geq 1$.
Thus the crucial question arising now is whether the projective modules
$X_{kp^t}$ in the decomposition (\ref{eqn lie dec}) could possibly 
be $\{0\}$. The next lemma shows that this is not the case, which will
be essential for Theorem~\ref{thm:vertexLie} below. 
We remark that, by the above proof, we have $X_{kp^t}\neq\{0\}$ if and 
only if the idempotent $f_{kp^t}\in F\mathfrak{S}_{kp^t}$ is different
from zero. 
This should follow from a close inspection of the results in \cite{BS} and \cite{Br},
but unfortunately is not explicitly stated there. Hence we provide
a straightforward proof, based on 
calculations in \cite{BLT}.

\begin{lemma}\label{lemma:X}
Keep the notation as in Theorem~\ref{prop:BE}.
Then $X_{kp^t}\neq \{0\}$, for all $t\geq 0$.
\end{lemma}

\begin{proof}
For $t\geq 0$, write
$$x_{kp^t}:=\frac{\dim(X_{kp^t})}{\dim(\LieF(kp^t))}=\frac{\dim(X_{kp^t})}{(kp^t-1)!}\,.$$
Note that, taking $d=t$ in Theorem~\ref{prop:BE}, we see that $X_{kp^t}$ is isomorphic to a submodule
of the Lie module $\LieF(kp^t)$.
Thus $0\leq x_{kp^t}\leq 1$ and it suffices to show that $x_{kp^t}>0$.  Now observe that, in the notation
of \cite[page 851]{BLT}, we have $X_{kp^t}=C(kp^t)$. Hence, by \cite[(10)]{BLT},
$$x_{kp^t}=1-\sum_{i=1}^ta_i'(x_{kp^{t-i}})^{p^i}\,,$$
where $a_i'=(kp^{t-i})^{-(p^i-1)}$. Since $k>1$, we have
$$\sum_{i=1}^ta_i'(x_{kp^{t-i}})^{p^i}\leq \sum_{i=1}^ta_i'\leq \sum_{i=1}^tk^{-(p^i-1)}<\sum_{j=1}^\infty k^{-j}=(k-1)^{-1}\leq 1\,.$$
Therefore, $x_{kp^t}>0$.
\end{proof}

We can now formulate the following reduction result:

\begin{thm}\label{thm:vertexLie}
Let $n=k\cdot p^d$, for some $d\geq 0$ and some $k>1$ with $p\nmid k$.

\medskip

{\rm (a)}\,  
Let $L$ be an indecomposable direct summand of $\LieF(n)$ with vertex $Q$.
There are some integer $t\in\{0,\ldots,d\}$, a composition $\lambda$ of $p^{d-t}$, 
and an indecomposable direct summand $L'$ of 
$\res_{\mathfrak{S}_\lambda}^{\mathfrak{S}_{p^{d-t}}}(\LieF(p^{d-t}))$
such that
$$ Q\leq_{\mathfrak{S}_n}(Q')^\sharp\leq 
\mathfrak{S}_{\lambda}^\sharp\leq 
\mathfrak{S}_{kp^{t}}\wr\mathfrak{S}_{p^{d-t}}\leq \mathfrak{S}_n\,,
\quad\text{for every vertex $Q'$ of $L'$}\,.$$

\medskip

{\rm (b)}\, In the situation of part~{\rm (a)}, 
if $S'$ is a $Q'$-source of $L'$, then 
there is a $Q$-source $S$ of $L$ such that 
$S\mid \res_{Q}^{(Q')^\sharp}((S')^\sharp)$.
Moreover,
there is an indecomposable direct summand $K$ of $\LieF(p^{d-t})$ having
a vertex $R$ with $Q'\leq R$.
Furthermore, there is an $R$-source $T$ of $K$ such that
$S'\mid \res^{R}_{Q'}(T)$.

\medskip

{\rm (c)}\, Conversely, let $0\leq t\leq d$, and let $K$
be any indecomposable direct summand of $\LieF(p^{d-t})$ with vertex $R$.
Then there is an indecomposable direct summand $L$ of $\LieF(n)$ 
with vertex $R^\sharp$, and every $R$-source of $K$ is then also
an $R^\sharp$-source of $L$.
\end{thm}

\begin{proof}
Parts (a) and (b) are immediate consequences of \ref{noth:vertex}, 
Theorem~\ref{thm:wreathreduce} and Theorem~\ref{prop:BE}.
Note that here we need the fact that the $F\mathfrak{S}_{kp^t}$-modules
$X_{kp^t}$ in Theorem~\ref{prop:BE} are projective.

It remains to prove (c). So let $t\in\{0,\ldots,d\}$.
Let further $X$ be any indecomposable direct summand of the 
$F\mathfrak{S}_{kp^t}$-module $X_{kp^t}$; 
note that here we need Lemma~\ref{lemma:X} to ensure that all
the projective modules in Theorem~\ref{prop:BE} are indeed non-zero.
Now consider the one-part partition $\lambda=(p^{d-t})$ of $p^{d-t}$.
Then, by \ref{noth:L}, we get the indecomposable
$F[\mathfrak{S}_{kp^t}\wr\mathfrak{S}_{p^{d-t}}]$-module 
$L':=X^{\otimes p^{d-t}}\otimes \widehat{K}$.
By Theorem~\ref{thm:wreathreduce}(c), $L'$ has vertex $R^\sharp$, and 
every $R$-source $T$ of $K$ yields the $R^\sharp$-source $T^\sharp$ of $L'$.
As we have remarked in \ref{noth:vertex}(d),
there is an indecomposable direct summand $L$ of
$\ind_{\mathfrak{S}_{kp^t}\wr\mathfrak{S}_{p^{d-t}}}^{\mathfrak{S}_n}(L')$ 
with vertex $R^\sharp$ and $R^\sharp$-source $T^\sharp$. 
By Theorem~\ref{prop:BE}, we have
$\ind_{\mathfrak{S}_{kp^t}\wr
 \mathfrak{S}_{p^{d-t}}}^{\mathfrak{S}_n}(L')\mid \LieF(n)$,
and hence assertion (c) follows.
\end{proof}

Exploiting our computational data from Section \ref{sec:comp}
and Theorem \ref{thm:vertexLie} we obtain the following

\begin{cor}\label{cor:smallLievertex2}
Let $p=2$ and $n=k\cdot 2^d$, where $k\geq 1$ is odd and $0\leq d\leq 3$,
and let $L$ be an indecomposable direct summand of $\LieF(n)$. 

\medskip

\quad {\rm (a)}\, 
Let $Q\leq\mathfrak{S}_n$ be a vertex of $L$, and let
$S$ be a $Q$-source. Then $Q$ is elementary abelian of order 
$|Q|\leq 2^d$, and $S$ is an endo-permutation module.
\medskip

\quad {\rm (b)}\, 
Suppose that $|Q|$ is maximal amongst the orders of the vertices of 
all the indecomposable direct summands of $\LieF(n)$. Then
one has 
$Q=_{\mathfrak{S}_n} E_{2^d}^\sharp\leq \mathfrak{S}_{2^d}^\sharp\leq
\mathfrak{S}_k\wr\mathfrak{S}_{2^d}\leq \mathfrak{S}_n$; 
in particular, $Q$ is uniquely determined up to $\mathfrak{S}_n$-conjugation.
Moreover, every $E_{2^d}$-source of $\LieFpf(2^d)$ 
is an $E_{2^d}^\sharp$-source of $L$.
\end{cor}

\begin{proof}
To show (a), by Theorem~\ref{thm:vertexLie} and \ref{noth:vertex}(c),
there is some integer $t\in\{0,\ldots,d\}$, and there is an indecomposable direct summand 
$L'$ of $\LieF(2^{d-t})$ with vertex $R$ such that 
$Q\leq_{\mathfrak{S}_n}R^\sharp\leq 
\mathfrak{S}_{k\cdot 2^t}\wr\mathfrak{S}_{2^{d-t}}\leq \mathfrak{S}_n$.
(Note that the image of the map $(-)^\sharp$
depends on the particular choice of $t$.)
Moreover, we observe from the results of \ref{noth:lie4}, \ref{noth:lie8} 
and Remark \ref{rem:lievertex} that all the indecomposable direct summands 
of $\LieF(2^{d-t})$ are either projective,
or are isomorphic to $\LieFpf(2^{d-t})$ and have 
elementary abelian vertex of order $2^{d-t}$
and endo-permutation sources. Finally, recall that the property of being 
an endo-permutation module is retained under restriction to subgroups 
and under taking direct summands. 

To show (b), recall again that if $L$ is an indecomposable direct summand 
of $\LieF(n)$ then there is some $t\leq d$ such that the vertices of $L$ are conjugate 
to subgroups of $E_{2^{d-t}}^\sharp$.
Now note  that, by Theorem~\ref{thm:vertexLie}(c), there
indeed is an indecomposable direct summand of $\LieF(n)$ 
having a vertex that is $\mathfrak{S}_n$-conjugate to $E_{2^d}^\sharp$.
Thus, if $|Q|$ is maximal then we have $|Q|=2^d$,  and
$Q$ is $\mathfrak{S}_n$-conjugate to $E_{2^d}^\sharp$; hence we may
assume that $Q=E_{2^d}^\sharp$. But this forces $t=0$ and $K=\LieFpf(2^d)$.
By Theorem~\ref{thm:vertexLie}(c) again, every $E_{2^d}$-source of
$\LieFpf(2^d)$ is an $E_{2^d}^\sharp$-source of $L$.
\end{proof}

The following result deals with the case where $p=3$. The proof is
completely analogous to that of Corollary~\ref{cor:smallLievertex2}, 
and is thus left to the reader.

\begin{cor}\label{cor:smallLievertex3}
Let $p=3$ and $n=k\cdot 3^d$, where $k\geq 1$ is such that $3\nmid k$,
and $0\leq d\leq 2$, and 
let $L$ be an indecomposable direct summand of $\LieF(n)$.

\medskip

\quad {\rm (a)}\, 
Let $Q\leq\mathfrak{S}_n$ be a vertex of $L$, and let
$S$ be a $Q$-source. Then $Q$ is elementary abelian of order
$|Q|\leq 3^d$, and $S$ is an endo-permutation module.

\medskip

\quad {\rm (b)}\, 
Let $|Q|$ be maximal amongst the orders of the vertices of
all the indecomposable direct summands of $\LieF(n)$. Then one has
$Q=_{\mathfrak{S}_n}E_{3^d}^\sharp\leq \mathfrak{S}_{3^d}^\sharp\leq
\mathfrak{S}_k\wr\mathfrak{S}_{3^d}\leq \mathfrak{S}_n$; 
in particular, $Q$ is uniquely determined up to $\mathfrak{S}_n$-conjugation.
Moreover, every $E_{3^d}$-source of $\LieFpf(3^d)$
is an $E_{3^d}^\sharp$-source of $L$.
\end{cor}


{\sc R.B.: School of Mathematics, University of Manchester, \\
Oxford Road, Manchester, M13 9PL, United Kingdom\\}
{\sf roger.bryant@manchester.ac.uk}

\smallskip

{\sc S.D.: Department of Mathematics,
University of Kaiserslautern,\\
P.O. Box 3049,
67653 Kaiserslautern,
Germany}\\
{\sf danz@mathematik.uni-kl.de}

\smallskip

{\sc K.E.: 
Mathematical Institute,
University of Oxford,\\
24--29 St Giles',
Oxford,
OX1 3LB,
United Kingdom,}\\
{\sf erdmann@maths.ox.ac.uk}

\smallskip

{\sc J.M.:
Lehrstuhl D f\"ur Mathematik, RWTH Aachen \\
Templergraben 64, D-52062 Aachen, Germany} \\
{\sf Juergen.Mueller@math.rwth-aachen.de}

\end{document}